\documentclass[11pt,a4paper]{article}
\usepackage{ucs}
\usepackage[applemac]{inputenc}
\usepackage[T1]{fontenc}
\usepackage[english]{babel}		
\usepackage{
	amsmath,
	amsthm,
	amssymb,
	amsfonts,
	mathtools,
	framed,
	graphicx,
	fixme
}
\usepackage[dvips]{hyperref}
\hypersetup{	
	colorlinks,
	linkcolor=black,
	citecolor=black,	urlcolor=black,
	breaklinks=true
}
\DeclareMathAlphabet{\mathpzc}{OT1}{pzc}{m}{it}
\usepackage{geometry}
\newtheoremstyle{lemma}{\topsep}{\topsep}
	{\itshape}
	{}
	{\bfseries}
	{.}
	{\newline}
	{\thmname{#1}\thmnumber{ #2}\thmnote{ #3}}	
\theoremstyle{lemma}
	\newtheorem{lemma}{Lemma}[section]
	\newtheorem{proposition}[lemma]{Proposition}
	
	\newtheorem{theorem}[lemma]{Theorem}

\newtheoremstyle{definition}{\topsep}{\topsep}
	{}
	{}
	{\bfseries}
	{.}
	{\newline}
	{\thmname{#1}\thmnumber{ #2}\thmnote{ #3}}	
\theoremstyle{definition}
	
	\newtheorem{remark}[lemma]{Remark}

\newcommand{\N}{\ensuremath{\mathbb{N}}}

\newcommand{\R}{\ensuremath{\mathbb{R}}}

\newcommand{\Sphere}{\ensuremath{\mathbb{S}}}
\newcommand{\HM}{\ensuremath{\mathcal{H}}}
\DeclarePairedDelimiter\abs{\lvert}{\rvert}
\DeclarePairedDelimiter\norm{\lVert}{\rVert}

\newcommand{\dd}{\ensuremath{\,\mathrm{d}}}	

\newcommand{\ds}{\ensuremath{\,\mathrm{d}s}}
\newcommand{\dy}{\ensuremath{\,\mathrm{d}y}}

\newcommand{\dx}{\ensuremath{\,\mathrm{d}x}}
\newcommand{\dt}{\ensuremath{\,\mathrm{d}t}}

\newcommand{\distH}{\ensuremath{d_{\mathcal{H}}}}
\DeclareMathOperator{\dist}{dist}
\DeclareMathOperator{\reach}{reach}
\DeclareMathOperator{\Tan}{Tan}
\DeclareMathOperator{\Nor}{Nor}
\DeclareMathOperator{\dual}{dual}
\DeclareMathOperator{\epi}{epi}

\DeclareMathOperator{\Unp}{Unp}
\DeclareMathOperator{\Hess}{Hess}
\DeclareMathOperator{\argmin}{argmin}
\DeclareMathOperator{\exterior}{ext}
\DeclareMathOperator{\interior}{int}
\DeclareMathOperator{\diam}{diam}

\renewcommand{\phi}{\varphi}
\renewcommand{\epsilon}{\varepsilon}

\hypersetup{
	pdftitle={On hypersurfaces of positive reach, alternating Steiner formul{\ae} and Hadwiger's Problem},
	pdfsubject={},
	pdfauthor={Sebastian Scholtes},
	pdfkeywords={},
	pdfcreator={},
	pdfproducer={}
}

\makeindex
\begin{document}

\title{On hypersurfaces of positive reach, alternating Steiner formul{\ae} and Hadwiger's Problem}
\author{\href{mailto:sebastian.scholtes@rwth-aachen.de}{Sebastian Scholtes}}
\date{\today}
\maketitle

\begin{abstract}
	We give new characterisations of sets of positive reach and show that a closed hypersurface has positive reach if and only if it is of class $C^{1,1}$.
	These results are then used to prove new alternating Steiner formul{\ae} for hypersurfaces of positive reach. Furthermore, it will turn out that every hypersurface that satisfies an
	alternating Steiner formula has positive reach. Finally, we provide a new solution to a problem by Hadwiger on convex sets and prove long time existence for the gradient flow
	of mean breadth.
\end{abstract}
\centerline{\small Mathematics Subject Classification (2000): 53A07, 52A20}

\section{Introduction}

In his seminal paper \cite{Federer1959a} Federer introduced the notion of sets of positive reach. Roughly speaking, the \emph{reach} of a closed set $A$ is the largest $s\geq 0$ such that all points 
whose distance to $A$ is smaller than $s$ possess a unique nearest point in $A$. Sets of positive reach share many of the properties that make convex sets so interesting and important,
but it is a much broader class. All closed convex sets as well as all closed $C^{2}$ submanifolds of $\R^{n}$ have positive reach in particular.
One of Federer's main results is a Steiner formula for sets of positive reach.
In the simplest case this means that for $A\subset \R^{n}$ closed and $0\leq s<\reach(A)$ the volume $V(A_{s})\vcentcolon=\HM^{n}(A_{s})$ of the parallel set is a polynomial of degree at most $n$. 
More precisely, there are real numbers $W_{k}(A)$, $k=0,\ldots, n$, such that
\begin{align}\label{steinerformulaintroduction}
	V(A_{s})=\sum_{k=0}^{n}{n \choose k}W_{k}(A)s^{k}
\end{align}
for $0\leq s< \reach(A)$ \cite[5.8 Theorem]{Federer1959a}. Here, the \emph{parallel set} of a non-empty set $A\subset\R^{n}$ is defined by
\begin{align*}
	A_{s}\vcentcolon=
		\begin{cases}
			\{x\in\R^{n}\mid \dist(x,A)\leq s\},&s\geq 0,\\
			\{x\in A\mid \dist(x,\partial A)\geq -s\},&s<0.
		\end{cases}
\end{align*}
In case of convex sets the $W_{k}$ are called \emph{querma{\ss}integrals} and in the more general context of sets with positive reach \emph{total curvatures} (although the total curvatures differ from the $W_{k}$
by a multiplicative constant depending on $n$ and $k$ and are usually numbered in reverse order).
These are important geometric quantities that characterise the sets involved. For example, for a non-empty compact set $A$ with positive reach we have $W_{0}(A)=\HM^{n}(A)$,
$W_{n}(A)=\chi(A)\HM^{n}(B_{1}(0))$ (see \cite[5.19 Theorem]{Federer1959a}); for $n\geq 2$ holds $W_{1}(A)=n^{-1}\mathcal{SM}(A)$ and
if additionally $A$ is convex and has non-empty interior we even have $W_{1}(A)=n^{-1}\HM^{n-1}(\partial A)$.\footnote{For an example of a compact set $A\subset \R^{2}$ of positive reach with 
$2^{-1}\HM^{1}(\partial A)<W_{1}(A)$ see \cite[Example 1]{Ambrosio2008a}.}
Here, $\chi(A)$ is the Euler-Poincar\'e characteristic of $A$ and $\mathcal{SM}(A)$ is the outer Minkowski content of $A$, for a definition see \cite{Ambrosio2008a}. 
In case of sets $A\subset \R^{n}$ of positive reach whose boundaries are of class $C^{1,1}$ 
the querma{\ss}integrals can also be written as mean curvature integrals, that is, as an integral over $\partial A$ of certain combinations of the classical principal curvatures that exist a.e.
(see Lemma \ref{quermassintegralsasmeancurvatureintegrals}); this is what the title of Federer's paper alludes to.\\

There are different characterisations of the reach of a set. For example, it can be defined as the largest $t$ such that two normals do not intersect in $A_{s}$ for all $s<t$ 
(see Lemma \ref{alternativecharacterisationreach} and for the definition of normals in this context \eqref{definitionnormalcone}).
In Theorem \ref{positivereachandsteinerformula} we give two new characterisations of sets of positive reach. The first tells us that a set has positive reach if and only if 
the set and its outer parallel sets satisfy an alternating Steiner formula. By \emph{alternating} we mean that the Steiner formula not only gives the volume of the outer parallel sets
(in our case $(A_{s})_{t}$ for $t\geq 0$), as in Federer's case, but the same polynomial also describes the volume of the inner parallel sets ($t<0$ is admissible). The second characterisation says
that a set has positive reach if and only if the parallel sets exhibit a semigroup-like structure.

\begin{theorem}[(Characterisation of sets of positive reach)]\label{positivereachandsteinerformula}
	Let $A\subset \R^{n}$ closed, $A\not\in \{\emptyset,\R^{n}\}$ and $r>0$. Then the following are equivalent
	\begin{itemize}
		\item
			for all $s\in (0,r)$ there are $W_{k}(A_{s})\in\R$ such that for $0<s+t<r$ holds
			\begin{align*}
				V((A_{s})_{t})=\sum_{k=0}^{n}{n\choose k}W_{k}(A_{s})t^{k},
			\end{align*}
		\item
			$(A_{s})_{t}=A_{s+t}$ for all for all $s\in (0,r)$ and $0<s+t<r$,
		\item
			$\reach(A)\geq r$.
	\end{itemize}
\end{theorem}
By means of the example $A\vcentcolon=[-b,b]^{2}\backslash [-a,a]^{2}$ for $0<a<b$, where
\begin{align*}
	V(A_{s})=4(b^{2}-a^{2})+8(b+a)s+(\pi-4)s^{2}\quad\text{for }0\leq s\leq a,
\end{align*}
see Figure \ref{boxannulus}, we find that it is essential to have the Steiner formula for the outer parallel sets, in order to characterise sets of positive reach.\\

\begin{figure}\label{boxannulus}
	\centering 
	\includegraphics[width=0.5\textwidth]{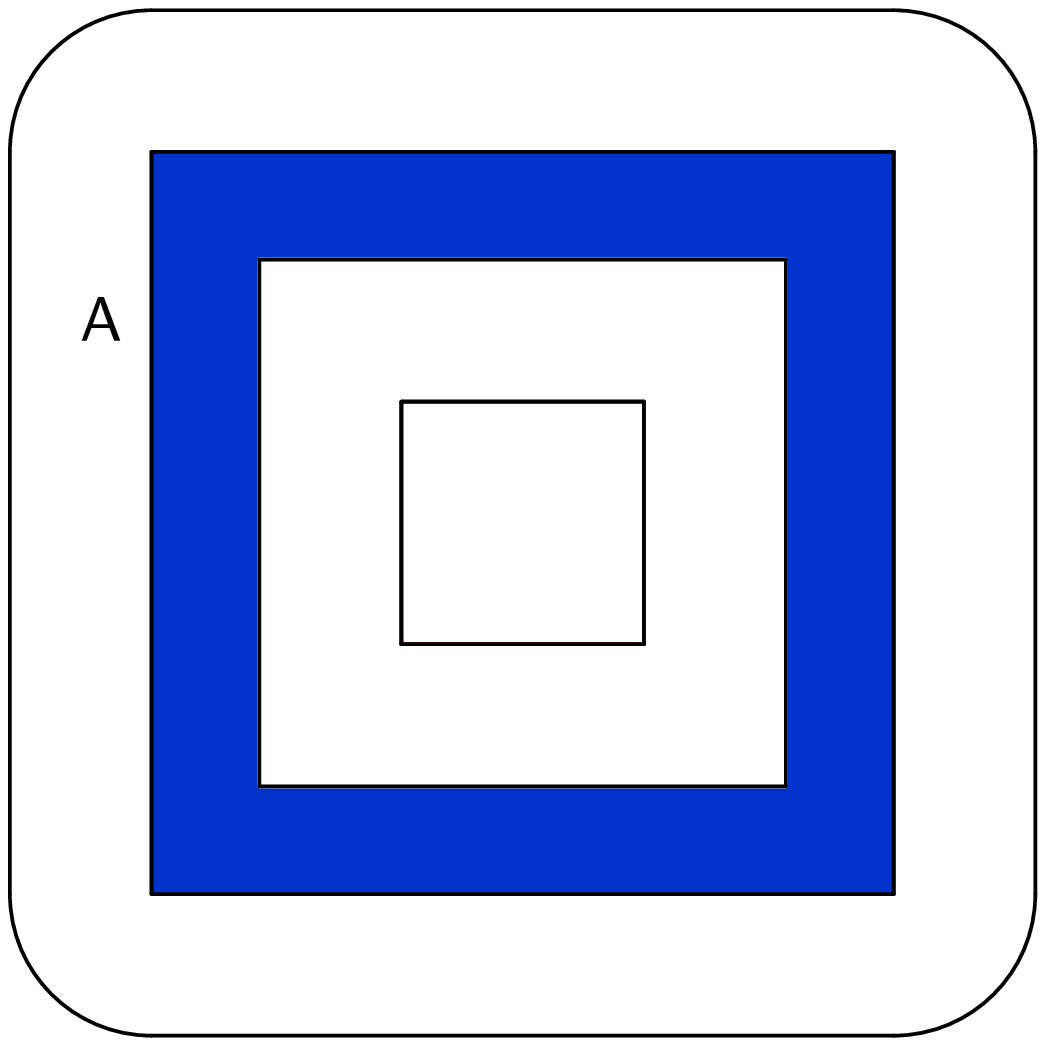}
	\caption{The set $A\vcentcolon=[-b,b]^{2}\backslash [-a,a]^{2}$ with outer parallel set.}
\end{figure}

As we have seen before, a set of positive reach possesses a Steiner formula \eqref{steinerformulaintroduction} for $0\leq s<\reach(A)$. Now, it is an obvious question to ask wether or not this formula can also
be extended to the inside of the set, i.e. if there is $u<0$ such that \eqref{steinerformulaintroduction} also holds for $u<s<\reach(A)$. Disappointingly, the answer is, in general and even for convex bodies: No! 
This can easily be seen by $A\vcentcolon= [-1,1]^{2}$, because $V(A_{s})=4+8s+\pi s^{2}$ for $s\geq 0$ but $V(A_{s})=(2+2s)^{2}=4+8s+4s^{2}$ for $s\in (-1,0)$, or by the example of the semi-circle,
where the formula for the volume of the inner parallel bodies is not even a polynomial (see \cite[Example 2]{Kocak2012a}). In \cite{Hernandez-Cifre2010a} a conjecture by 
Matheron, that the volume of the inner parallel bodies of a convex set is bounded below by the Steiner polynomial, is disproven and conditions for different bounds on the volume of the inner parallel bodies
are given. This line of research was continued in \cite{Hernandez-Cifre2010c}. Furthermore \cite{Kocak2012a} showed that the volume of the inner parallel bodies of a polytope in $\R^{n}$ is, what the authors called,
a \emph{degree $n$ pluriphase Steiner-like function}, which basically allows the querma{\ss}integrals to change their values at a finite number of points. In Theorem \ref{theoremsteinerformulaiffpoasitivereach}
we characterise closed sets whose inner and outer parallel sets posses an alternating Steiner formula as those sets of this class whose boundaries have positive reach.
\begin{theorem}[(Alternating Steiner formula and reach of the boundary)]\label{theoremsteinerformulaiffpoasitivereach}
	Let $A\subset \R^{n}$ be closed and bounded by a closed hypersurface, $r>0$. Then the following are equivalent
	\begin{itemize}
		\item
			for all $s\in (-r,r)$ there are $W_{k}(A_{s})\in\R$ such that for $-r<s+t<r$ holds
			\begin{align*}
				V((A_{s})_{t})=\sum_{k=0}^{n}{n\choose k}W_{k}(A_{s})t^{k},
			\end{align*}
		\item
			$(A_{s})_{t}=A_{s+t}$ for all for all $s\in (-r,r)$ and $-r<s+t<r$,
		\item
			$\reach(\partial A)\geq r$,
		\item
			$\partial A$ is a closed $C^{1,1}$ hypersurface with $\reach(\partial A)\geq r$.
	\end{itemize}
\end{theorem}

To prove this theorem we need a characterization of closed hypersurfaces of positive reach. 
By a \emph{closed hypersurface} in $\R^{n}$ we mean a \emph{topological sphere}, that is, the homeomorphic image of $\Sphere^{n-1}$.

\begin{theorem}[(Closed hypersurfaces have positive reach iff $C^{1,1}$)]\label{equivalencepositivereachC11}
	Let $A$ be a closed hypersurface in $\R^{n}$. Then $A$ has positive reach if and only if $A$ is a $C^{1,1}$ manifold.
\end{theorem}
This result was already featured in \cite[\S4 Theorem 1]{Lucas1957a}, a reference that is not easily accessible and which does not seem to be widely known.
Clearly, the result was  stated in a slightly different form, as Federer had not coined the term reach yet and is also proven by different methods. 
In resources more readily available, we find the direction $\reach(A)>0$ implies $C^{1,1}$
in \cite[Proposition 1.4]{Lytchak2005a} and \cite[Theorem 1.2]{Howard2010a}. The other direction can, other than \cite[\S4 Theorem 1]{Lucas1957a}, only be found as a remark without proof, for example 
in \cite[below 2.1 Definitions]{Fu1989a} or \cite[under Theorem 1.1]{Lytchak2004b}. Another hint to this result may be found in \cite[4.20 Remark]{Federer1959a}.
Considering that Theorem \ref{equivalencepositivereachC11} is mostly folklore and a uniform proof of both directions
together is not available it seems to be worth to give a detailed proof of this result. 
To show this, we use a characterisation of $C^{1,\alpha}_{\mathrm{loc}}$ functions, Proposition \ref{characterisationofh\"olderfunctions}, 
which states that a function is of class $C^{1,\alpha}_{\mathrm{loc}}$ if and only if
\begin{align*}
	\abs{f(x-h)-2f(x)+f(x+h)}\leq C \abs{h}^{1+\alpha}.
\end{align*}
One direction of this characterisation is mostly taken from \cite[Lemma 2.1]{Calabi1970a}, but since we suspect that it might be useful in other contexts, too, it deserves an elaborate proof.

To some extent Theorem \ref{equivalencepositivereachC11} can be thought of as a generalization of
\cite[Lemma 4]{Cantarella2002a}, \cite[Lemma 2]{Gonzalez2002b}, \cite[Theorem 1 (iii)]{Schuricht2003a} and \cite[Theorems 5.1 and 5.2]{Strzelecki2006a} to higher dimension 
(although the codimension is not restricted to one).
There, different notions of thickness, specific to either curves or surfaces, were investigated and sets of positive thickness were
characterized. These notions of thickness are equal to the reach of the curves and surfaces under consideration.\\

The problem of characterising convex sets whose querma{\ss}integrals are differentiable, is known as \emph{Hadwiger's problem} \cite{Hadwiger1955a}. 
To be more precise, denote by $\mathcal{K}^{n}$ the class of non-empty compact convex sets in $\R^{n}$ and  by $\mathcal{R}_{p}(r)$, for $r\geq 0$ and 
$0\leq p\leq n-1$, the class of all $K\in\mathcal{K}^{n}$ such that $\phi_{i}: (-r,\infty)\to \R,\, s\mapsto W_{i}(K_{s})$ for $i=0,\ldots, p$ are differentiable with
$W_{i}'(s)=(n-i)W_{i+1}(s)$, where we abbreviate $W_{i}(s)=W_{i}(K_{s})$. 
In \cite[Theorem 1.1]{Hernandez-Cifre2010b} the class $\mathcal{R}_{n-1}$ of convex sets $K$ whose querma{\ss}integrals are differentiable on $(-r(K),\infty)$, where $r(K)$ is the \emph{inradius}, 
is identified as the set of outer parallel bodies of lower dimensional convex sets, i.e. 
\begin{align}\label{Rn-1characterisationouterparallelbodies}
	\mathcal{R}_{n-1}=\{L_{s}\mid L\in\mathcal{K}^{n}, \dim(L)\leq n-1, s\geq 0\},
\end{align}
and \cite{Hernandez-Cifre2011a} gives a characterisation of $\mathcal{R}_{n-2}$ of a more complicated nature.\footnote{Actually, these charaterisations were done in a more general setting, which
not only considers parallel sets, which are Minkowski sums with balls, but also allows for Minkowski sums with a certain class of convex sets.}
Using our results of the present paper we can give the following new characterisation of the class $\mathcal{R}_{n-1}(r)$.
\begin{theorem}[(Characterisation of $\mathcal{R}_{n-1}(r)$)]\label{bigequivalence}
	Let $K\in \mathcal{K}^{n}$, $r>0$.
	Then the following are equivalent
	\begin{itemize}
		\item
			$K\in \mathcal{R}_{n-1}(r)$,
		\item
			there is a convex $L$ with $K=L_{r}$,
		\item
			$K=(K_{-r})_{r}$,
		\item
			$\reach(\partial K)\geq r$,
		\item
			$\partial K$ is a closed $C^{1,1}$ hypersurface with $\reach(\partial K)\geq r$.
	\end{itemize}
\end{theorem}

Additionally, these results give us a long time existence result for the energy dissipation equality (EDE) gradient flow of the mean breadth $W_{n-1}$ on the space $\mathcal{K}^{1,1}$, of all sets in $\mathcal{K}^{n}$ 
with non-empty interior and $C^{1,1}$ boundary, equipped with the Hausdorff distance $d_{\HM}$. For the essential notation see the beginning of Section \ref{sectiongradientflow} and for 
more detailed information on gradient flows on metric spaces we refer to \cite{Ambrosio2005a}.

\begin{proposition}[(Gradient flow of the mean breath $W_{n-1}$ on $(\mathcal{K}^{1,1},\distH)$)]\label{gradientflowmeanbreadth}
	Let $K\in \mathcal{K}^{1,1}$ and $T\vcentcolon=\omega_{n}^{-1}\reach(\partial K)$ then
	\begin{align*}
		x:[0,T)\to\mathcal{K}^{1,1},\, t\mapsto K_{-\omega_{n} t}
	\end{align*}
	is a gradient flow in the (EDE) sense for $W_{n-1}$ on $(\mathcal{K}^{1,1},\distH)$, i.e.
	\begin{align}\label{gradientflowede}
		W_{n-1}(x(t))+\frac{1}{2}\int_{s}^{t}\abs{\dot x(u)}^{2}\dd u+\frac{1}{2}\int_{s}^{t}\abs{\nabla W_{n-1}}^{2}(x(u))\dd u=W_{n-1}(x(s))
	\end{align}
	for all $0\leq s\leq t<T$ and $x$ is an absolutely continuous curve. Additionally, $x(t)\to x(T)$ in $\distH$ for $t\to T$, where $x(T)\vcentcolon=K_{-\reach(\partial K)}$, and $x(T)$ 
	is either a convex set contained in an affine $n-1$ dimensional space or a convex set with non-empty interior with $\reach(\partial x(T))=0$.
\end{proposition}

By $\omega_{n}$ we denote the $n$-dimensional volume of the unit ball in $\R^{n}$, i.e. $\omega_{n}\vcentcolon=\HM^{n}(B_{1}(0))$.

\subsection*{Acknowledgements}

The author wishes to thank his advisor Heiko von der Mosel for his interest and encouragement as well as reading the manuscript and making many helpful suggestions. 
Additionally, we thank the participants of the ``2\textsuperscript{nd} Workshop on Geometric Curvature Energies''
2012 in Steinfeld, Germany, for valuable comments and remarks, and finally Rolf Schneider for his detailed answer to a particular question on querma{\ss}integrals.

\section{Sets of positive reach}\label{sectionpositivereach}

As a generalisation of convex sets Federer introduced in his seminal paper \cite{Federer1959a} the notion of sets of positive reach. A closed set $A\subset \R^{n}$ is said to be of \emph{reach} $t$ at a point $a\in A$, 
denoted by $\reach(A,a)=t$, if $t$ is the supremum of all $\rho>0$ such that the restriction $\tilde\xi_{A}\vert_{B_{\rho}(a)}$ of the metric projection map
\begin{align*}
	\tilde\xi_{A}:\R^{n}\to \mathcal{P}(A),\,x\mapsto \{a\in A\mid \abs{x-a}=\dist(x,A)\}
\end{align*}
is single valued, or to be more precise, singleton valued. Here, $\mathcal{P}(A)$ denotes the \emph{power set} of $A$. The \emph{reach of a set} $A$ is then defined to be $\reach(A)\vcentcolon=\inf_{a\in A}\reach(A,a)$. 
By $\Unp(A)$ we denote the set of all points that have a unique nearest point in $A$, that is
\begin{align*}
	\Unp(A)\vcentcolon=\{x\in\R^{n}\mid \#\tilde\xi_{A}(x)=1\}.
\end{align*}
Now, we introduce another \emph{metric projection map} $\xi_{A}$, defined on $\Unp(A)$ so that $\tilde\xi_{A}(x)$ is already a singleton, by
\begin{align*}
	\xi_{A}:\Unp(A)\to A,\,x\mapsto \argmin_{a\in A}(\abs{x-a}).
\end{align*}
This is essentially the same mapping as before, but it ``extracts'' the unique nearest point from the singleton.\\

In what follows, we always assume $A\subset \R^{n}$, $A\not\in\{\emptyset,\R^{n}\}$, so that
we do not have to worry about certain pathologies. Especially, we have $\partial A\not=\emptyset$, because else we would have $\overline A=A\dot{\cup}\partial A$, but $A=\emptyset$ and $A=\R^{n}$ are the
only closed and open sets in $\R^{n}$. We also use $\dist(x,A)=\dist(x,\overline A)$ and for $x\not\in A$ additionally $\dist(x,A)=\dist(x,\partial A)$ without further notice.

\begin{lemma}[(Properties of $\tilde\xi_{A}$)]\label{projectionproperties}
	Let $A\subset \R^{n}$, $a\in A$. Then $a\in \tilde\xi_{A}(x)$ if and only if $\xi_{A}(x_{t})=a$ for $x_{t}\vcentcolon=a+t(x-a)$ and $t\in [0,1)$.
\end{lemma}
\begin{proof}
	\textbf{Step 1}
		Let $a\in \tilde\xi_{A}(x)$.
		Suppose there is $b\in A\backslash\{a\}$ with $\abs{x_{t}-b}\leq \abs{x_{t}-a}$ for a fixed $t\in [0,1)$. Then
		\begin{align}\label{strictinequalityab}
			\begin{split}
				\MoveEqLeft
				\abs{x-b}< \abs{x-x_{t}}+\abs{x_{t}-b}
				\leq\abs{x-x_{t}}+\abs{x_{t}-a}\\
				&=\abs{x-[a+t(x-a)]}+\abs{[a+t(x-a)]-a}\\
				&=(1-t)\abs{x-a}+t\abs{x-a}\\
				&=\abs{x-a},
			\end{split}
		\end{align}
		but this contradicts $a\in \tilde\xi_{A}(x)$. The strict inequality in \eqref{strictinequalityab} holds, because else we would have
		$b\in x+[0,\infty)(a-x)$, which is not compatible with $\abs{x_{t}-b}\leq \abs{x_{t}-a}$ and $\abs{a-x}\leq \abs{b-x}$.\\
	\textbf{Step 2}
		Let $\xi_{A}(x_{t})=a$ for $t\in [0,1)$ and assume that there is $b\in A\backslash\{a\}$, such that $\abs{b-x}<\abs{a-x}$. Then
		\begin{align*}
			2(1-t)\abs{x-a}+\abs{x-b}=2\abs{x_{t}-x}+\abs{x-b}<\abs{x-a}
		\end{align*}
		for $2^{-1}+2^{-1}\abs{x-b}/\abs{x-a}<t<1$, so that
		\begin{align*}
			\abs{x_{t}-b}\leq \abs{x_{t}-x}+\abs{x-b}<\abs{x-a}-\abs{x_{t}-x}=t\abs{x-a}=\abs{x_{t}-a},
		\end{align*}
		which contradicts our hypothesis.
\end{proof}

We define the \emph{tangent cone} of a set $A\subset\R^{n}$ at $a\in A$, to be
\begin{align*}
	\Tan_{a}A\vcentcolon=\Big\{tv\mid t\geq 0, \exists a_{k}\in A\backslash \{a\}:v=\lim_{k\to\infty}\frac{a_{k}-a}{\abs{a_{k}-a}}\Big\}\cup\{0\}
\end{align*}
and the \emph{normal cone} of $A$ at $a$ to be the dual cone of $\Tan_{a}A$, in other words
\begin{align}\label{definitionnormalcone}
	\Nor_{a}A\vcentcolon=\dual(\Tan_{a}A)
	=\{u\in\R^{n}\mid \langle u,v\rangle\leq 0\text{ for all }v\in \Tan_{a}A\}.
\end{align}
The normal cone is always a convex cone, while it may happen that the tangent cone is not convex. From \cite[4.8 Theorem (2)]{Federer1959a} we know that $\xi_{A}(x)=a$ implies
$x-a\in\Nor_{a}A$. Another representation of the normal cone
\begin{align}\label{normalconefederer12}
	\Nor_{a}A=\{tv\mid t\geq 0, \abs{v}=s,\xi_{A}(a+v)=a\}
\end{align} 
for $\reach(A,a)>s>0$ can be found in \cite[4.8 Theorem (12)]{Federer1959a}. Unfortunately, there seems to be a small gap at the very end of the proof of this item in Federer's paper.
Namely, it has not been taken into consideration that the cone $S$, which is set to be the
right-hand side of \eqref{normalconefederer12}, can a priori be empty. That this is indeed not the case is shown in Lemma \ref{projectionissurjectiveonUbackslashA}.
From \eqref{normalconefederer12} we infer
\begin{align}\label{tracingthenormal}
	x-a\in\Nor_{a}A,\, x\not=a\quad\Rightarrow\quad
	\xi_{A}(x_{s})=a\quad\text{ for }s<\reach(A,a)\text{ and }
	x_{s}=a+s\frac{x-a}{\abs{x-a}},
\end{align}
as $s\frac{x-a}{\abs{x-a}}\in \Nor_{a}A$, so that $v$ from \eqref{normalconefederer12} must be equal to $s\frac{x-a}{\abs{x-a}}$.

\begin{figure}
	\centering 
	\includegraphics[width=0.8\textwidth]{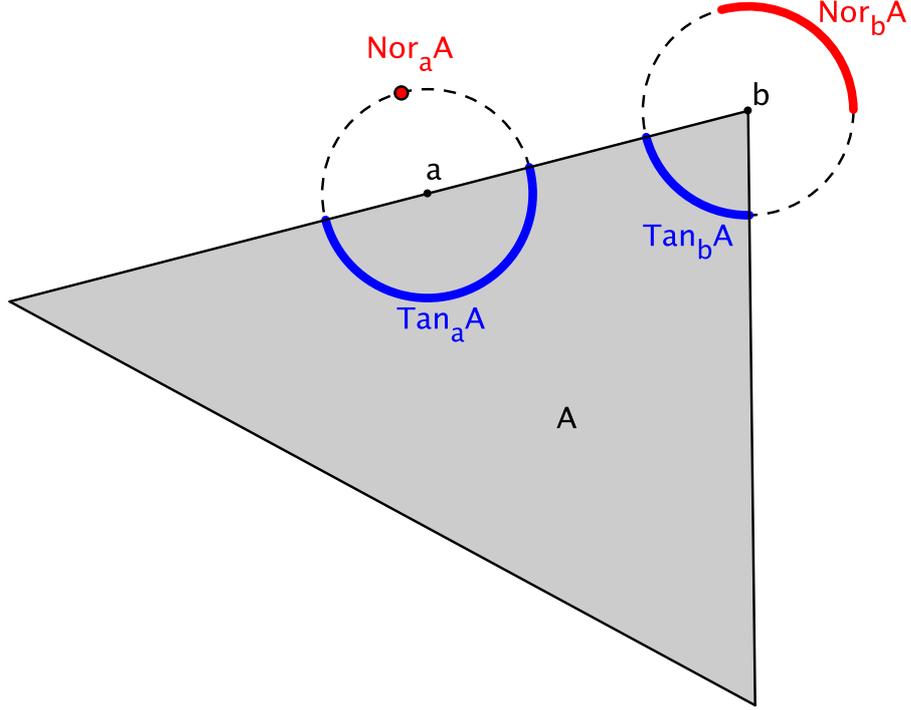}
	\caption{Directions in tangent and normal cone of a set $A$ at two different points.}
	\label{tangentnormalcone}
\end{figure}

\begin{lemma}[(If $\reach(A,a)>0$ then there is $v\in \Unp(A)\backslash \{a\}$ with $\xi_{A}(v)=a$)]\label{projectionissurjectiveonUbackslashA}
	Let $A\subset\R^{n}$ be a closed set, $a\in \partial A$ and $\reach(A,a)>0$. Then there is $v\in \Unp(A)\backslash \{a\}$ with $\xi_{A}(v)=a$.
\end{lemma}
\begin{proof}
	\textbf{Step 1}
		We adapt the proof of \cite[4.8 Theorem (11)]{Federer1959a} to our situation. Let $a\in \partial A$, $0<r<\reach(A,a)$, $0<\epsilon<\reach(A,a)-r$. 
		Without loss of generality we might assume that $a=0$. Then there is a sequence $u_{k}\in \Unp(A)\backslash A$, with $u_{k}\to a$, $\abs{u_{k}}< \epsilon/3$.
		For $\rho\in [0,r]$, $k\in\N$ and $\delta(x)\vcentcolon=\dist(x,A)$ set
		\begin{align*}
			\eta(u_{k},\rho)\vcentcolon&=\xi_{A}(u_{k})+\frac{\delta(u_{k})+\rho}{\delta(u_{k})}(u_{k}-\xi_{A}(u_{k})).
		\end{align*}
		Then
		\begin{align*}
			\abs{\eta(u_{k},\rho)}\leq 2\abs{u_{k}}+\delta(u_{k})+\rho\leq 3\abs{u_{k}}+ r \leq \epsilon +r<\reach(A,a),
		\end{align*}
		hence $\eta(u_{k},\rho)\in \Unp(A)$.\\
	\textbf{Step 2}
		Now, we want to show that $\xi_{A}(\eta(u_{k},\rho))=\xi_{A}(u_{k})$. Assume that this is not the case. Then
		\begin{align*}
			1\leq\tau \vcentcolon=\sup\{t>0\mid \xi_{A}[\xi_{A}(u_{k})+t(u_{k}-\xi_{A}(u_{k}))]=\xi_{A}(u_{k})\}
			\leq \frac{\delta(u_{k})+\rho}{\delta(u_{k})}<\infty,
		\end{align*}
		by Lemma \ref{projectionproperties}.
		Now, $\xi_{A}(u_{k})+\tau(u_{k}-\xi_{A}(u_{k}))\not\in \Unp(A)^{\circ}$, by \cite[4.8 Theorem (6)]{Federer1959a}, but
		\begin{align*}
			\MoveEqLeft
			\abs{\xi_{A}(u_{k})+\tau(u_{k}-\xi_{A}(u_{k}))}
			\leq \abs{\xi_{A}(u_{k})}+\tau \delta(u_{k})\\
			&\leq 2\abs{u_{k}}+\delta(u_{k})+\rho\leq 3\abs{u_{k}}+ r
			\leq \epsilon +r<\reach(A,a).
		\end{align*}
		\emph{Contradiction}.\\
	\textbf{Step 3}
		As $\abs{\eta(u_{k},r)}\leq \epsilon+r$ there must be a convergent subsequence, i.e. there is $v\in \R^{n}$ with
		\begin{align*}
			v=\lim_{l\to\infty}\eta(u_{k_{l}},r)\quad\text{and}\quad
			\abs{v}=\lim_{l\to\infty}\abs{\eta(u_{k_{l}},r)}=r,
		\end{align*}
		hence $v\in \Unp(A)\backslash \{a\}$ and according to Step 2 we have
		\begin{align*}
			\xi_{A}(v)=\lim_{l\to\infty}\xi_{A}(\eta(u_{k_{l}},r))=\lim_{l\to\infty}\xi_{A}(u_{k_{l}})=\xi_{A}(a)=a,
		\end{align*}
		since $\xi_{A}$ is continuous on $\Unp(A)$, see \cite[4.8 Theorem (4)]{Federer1959a}.
\end{proof}

Note that any closed hypersurface $A$ is compact and by the 
Jordan–Brouwer Separation Theorem it has a well-defined inside $\interior(A)$ and outside $\exterior(A)$.
From the definitions it is immediately clear that
\begin{align}\label{reachAminreachextint}
	\reach(A)=\min\{\reach(\overline{\interior(A)}),\reach(\overline{\exterior(A)})\}.
\end{align}

\begin{lemma}[(Alternative characterisation of $\reach$ I)]\label{alternativecharacterisationreach}
	Let $A\subset \R^{n}$ closed, $A\not\in\{\emptyset,\R^{n}\}$ and $\reach(A)>0$. Then
	\begin{align}\label{alternativerepresentationofreach}
		\reach(A)=\sup\{t\mid \forall a,b\in A, a\not=b:
		(a+\Nor_{a}A)\cap(b+\Nor_{b}A)\cap B_{t}(A)=\emptyset\}.
	\end{align}
\end{lemma}
\begin{proof}
	Let $a,b\in A$, $a\not= b$ and $u\in \Nor_{a}A$, $v\in \Nor_{b}A$ with $a+u=b+v$. Then by \eqref{tracingthenormal} we must have either $\abs{u}\geq \reach(A)$ or $\abs{v}\geq \reach(A)$, 
	because else $\xi_{A}(a+u)=a$ and $\xi_{A}(b+v)=b$ contradicts $a\not=b$. Hence $\reach(A)$ is not larger than the right-hand side of \eqref{alternativerepresentationofreach}.
	This means, for $\reach(A)=\infty$ we have proven the proposition. Let $\reach(A)<\infty$. Clearly, for $\epsilon>0$ there must be $a_{\epsilon}\in A$ and $u_{\epsilon}\in\Sphere^{n-1}$
	with $x_{\epsilon}=a_{\epsilon}+(\reach(A)+\epsilon)u_{\epsilon}\not\in\Unp(A)$. Hence, there are two different points $b_{\epsilon}\not= c_{\epsilon}$
	such that $b_{\epsilon},c_{\epsilon}\in \tilde\xi_{A}(x_{\epsilon})$. Therefore $x_{\epsilon}-b_{\epsilon}\in\Nor_{b_{\epsilon}}A$ and $x_{\epsilon}-c_{\epsilon}\in\Nor_{c_{\epsilon}}A$,
	see \cite[4.8 Theorem (2)]{Federer1959a}, i.e. $x_{\epsilon}\in (b_{\epsilon}+\Nor_{b_{\epsilon}}A)\cap (c_{\epsilon}+\Nor_{b_{\epsilon}}A)$ 
	and $\abs{x_{\epsilon}-b_{\epsilon}}=\abs{x_{\epsilon}-c_{\epsilon}}\leq \abs{x_{\epsilon}-a_{\epsilon}}=\reach(A)+\epsilon$. Consequently, the right-hand side of \eqref{alternativerepresentationofreach} cannot be larger than $\reach(A)$.
\end{proof}

\begin{lemma}[(Properties of parallel sets)]\label{propertiesofparallelsets}
	Let $A\subset \R^{n}$, $A\not\in\{\emptyset,\R^{n}\}$.
	\begin{enumerate}
		\item
			For $s>0$ holds $\partial [A_{s}]\subset\{x\in \R^{n}\backslash A\mid \dist(x,\partial A)=s\}$.
		\item
			For $s,t\geq 0$ holds $(A_{s})_{t}=A_{s+t}$.
		\item
			For $s\geq 0$ and $-s\leq t\leq 0$ holds $A_{s+t}\subset (A_{s})_{t}$.
		\item
			For $s<0$ holds $\partial [A_{s}]=\{x\in A\mid \dist(x,\partial A)=\abs{s}\}$.
		\item
			For $s,t\leq 0$ holds $(A_{s})_{t}=A_{s+t}$.
		\item
			For $s\leq 0$ and $0\leq t\leq -s$ holds $(A_{s})_{t}\subset A_{s+t}$.
	\end{enumerate}
\end{lemma}
\begin{proof}
	\textbf{(a)}
		Let $s>0$ and $x\in \partial [A_{s}]$. As $\dist(\cdot, A)$ is continuous, the set $A_{s}$ is closed and $\dist(x, A)\leq s$. 
		Therefore, for every $\epsilon>0$ there are points $y\in B_{\epsilon}(x)$ with $\dist(y,A)>s$. Hence, $x\not\in A$ and $\dist(x,A)=\dist(x,\partial A)=s$.\\
	\textbf{(b)}
		For $s=0$ or $t=0$ the equality is evident.
		Let $s,t> 0$. Then $A_{s}\subset A_{s+t}$ and for $x\in (A_{s})_{t}\backslash A_{s}$ we have
		\begin{align*}
			\dist(x,A)\leq \dist(x,\partial [A_{s}])+\dist(\partial[A_{s}],A)\leq t+s
		\end{align*}
		and hence $x\in A_{s+t}$. Clearly $A_{s}\subset (A_{s})_{t}$, therefore let $x\in A_{s+t}\backslash A_{s}$. Then there is $y\in \tilde\xi_{\partial A}(x)$ and there is $t_{0}\in [0,1)$, such that $\abs{z-y}=s$ for
		$z=y+t_{0}(x-y)$. Considering Lemma \ref{projectionproperties} we know that $z\in A_{s}$ and additionally we have
		\begin{align*}
			\abs{x-y}=\abs{x-z}+\abs{z-y}=\abs{x-z}+s\leq t+s,
		\end{align*}
		note that $x,y$ and $z$ are on a straight line with $z$ between $x$ and $y$. This means $\abs{x-z}\leq t$ and hence $x\in (A_{s})_{t}$.\\
	\textbf{(c)}
		Let $s\geq 0$, $-s\leq t\leq 0$ and $x\in A_{s+t}$. Then $x\in A_{s}$ and
		\begin{align*}
			-\dist(x,\partial [A_{s}])+s=-\dist(x,\partial [A_{s}])+\dist(\partial [A_{s}],A)\leq \dist(x,A)\leq s+t
		\end{align*}
		and hence $\dist(x,\partial[A_{s}])\geq -t$, i.e. $x\in (A_{s})_{t}$.\\
	\textbf{(d)}
		Let $s<0$ and $x\in \partial [A_{s}]$. As $\dist(\cdot,\partial A)$ is continuous the set $A_{s}$ is closed and $\dist(x,\partial A)\geq \abs{s}$. Then $x\in A_{s}$
		and for every $\epsilon>0$ there are points $y\in B_{\epsilon}(x)$ with $\dist(y,\partial A)<\abs{s}$. Hence $\dist(x,\partial A)=\abs{s}$.
		Now, let $x\in A$ with $\dist(x,\partial A)=\abs{s}$. Then $x\in A_{s}$. As $\partial A$ is closed there exists $a\in \tilde\xi_{\partial A}(x)$ and according to Lemma \ref{projectionproperties} we have
		$\dist(x_{t},\partial A)=t\abs{x-a}=t\abs{s}$ and hence $x_{t}\in \R^{n}\backslash A_{s}$ for $t\in (0,1)$. Consequently, $x\in \overline{\R^{n}\backslash A_{s}}$ and $x\in A_{s}$, therefore $x\in\partial [A_{s}]$.\\
	\textbf{(e)}
		For $s=0$ or $t=0$ the equality is evident.
		Let $s,t< 0$ and $x\in (A_{s})_{t}$. Then, as $\partial A$ is closed and non-empty, there is $y\in \tilde\xi_{\partial A}(x)$ and there is $t_{0}\in [0,1]$ such that $\abs{y-z}=\abs{s}$ for
		$z=y+t_{0}(x-y)$. Considering Lemma \ref{projectionproperties} we have
		$z\in A_{s}$. From $\abs{z-x}\geq \abs{t}$ we infer
		\begin{align*}
			\abs{s+t}=\abs{s}+\abs{t}\leq \abs{y-z}+\abs{z-x}=\abs{x-y}=\dist(x,\partial A),
		\end{align*}
		note that $x,y$ and $z$ are on a straight line with $z$ between $x$ and $y$. This means $x\in A_{s+t}$.
		Now let $x\in A_{s+t}$. Then $x\in A_{s}$ and
		\begin{align*}
			\abs{s+t}=\abs{s}+\abs{t}\leq\dist(x,\partial A)\leq \dist(x,\partial [A_{s}])+\dist(\partial[A_{s}],\partial A)
			=\dist(x,\partial [A_{s}])+\abs{s},
		\end{align*}
		by (d), so that $x\in (A_{s})_{t}$.\\
	\textbf{(f)}
		Let $s\leq 0$, $0\leq t\leq -s$ and $x\in (A_{s})_{t}$. Then $x\in \overline A$ and
		\begin{align*}
			-s-t\leq \dist(A_{s},\partial A)-\dist(x,A_{s})\leq \dist(x,\partial A),
		\end{align*}
		i.e. $x\in A_{s+t}$.
\end{proof}

The examples $\partial B_{1}(0)$, $\partial [0,1]^{2}$ and $[0,1]^{2}$ suffice to show that the inclusions in (a), (c) and (f), respectively, can be strict.

\begin{lemma}[(Alternative characterisation of $\reach$ II)]\label{alternativecharacterisationofreach2}
	Let $A\subset \R^{n}$ closed, $A\not\in\{\emptyset,\R^{n}\}$ and $r>0$. Then
	\begin{align}\label{alternativereachsemigroupproperty}
		\reach(A)\geq r\quad\Leftrightarrow\quad
		(A_{s})_{t}=A_{s+t}\text{ for all }s\in(0,r),\, t\in (-s,r-s).
	\end{align}
\end{lemma}
\begin{proof}
	\textbf{Step 1}
		Let $\reach(A)\geq r$. Let $s\in (0,r)$. For $t=0$ nothing needs to be shown.
		Let $t\in(0,r-s)$. We then always have $(A_{s})_{t}=A_{s+t}$, see Lemma \ref{propertiesofparallelsets} (b). 
		Let $s\in (0,r)$, $t\in (-s,0)$, then by Lemma \ref{propertiesofparallelsets} (c) we always have $A_{s+t}\subset (A_{s})_{t}$.
		For $x\in A$ we automatically have $x\in A_{s+t}$, so let $x\in (A_{s})_{t}\backslash A$. 
		As $\reach(A)\geq r$ we find a unique $y=\xi_{A}(x)$ and by \eqref{tracingthenormal}
		we additionally know $\dist(x_{u},A)=\abs{x_{u}-y}=u$ for $x_{u}\vcentcolon=y+u(x-y)/\abs{x-y}$, $u<r$. Then $x_{s}\in \partial [A_{s}]$, because $x_{u}\in A_{s}$ for $0\leq u \leq s$ 
		and $x_{u}\in\R^{n}\backslash A_{s}$ for $s<u<r$, so that 
		$\abs{x-x_{s}}\geq -t$, $\dist(x,A)=\abs{x-y}<s$ and hence
		\begin{align*}
			\dist(x,A)=\abs{x-y}=\abs{x_{s}-y}-\abs{x_{s}-x}\leq s+t,
		\end{align*}
		note that $y,x$ and $x_{s}$ are on a straight line with $x$ between $y$ and $x_{s}$. Hence $x\in A_{s+t}$.\\
	\textbf{Step 2}
		The other direction is a the contrapositive of Lemma \ref{differencesetcontainsinnerpoints} if we put $s=\sigma+\tau$ and $t=-\tau$.
\end{proof}

\begin{lemma}[(If $\reach(A)<r$ then $(A_{\sigma+\tau})_{-\tau}\backslash A_{\sigma}$ contains an inner point)]\label{differencesetcontainsinnerpoints}
	Let $A\subset \R^{n}$ be closed, $A\not\in\{\emptyset,\R^{n}\}$ and $\reach(A)<r$. Then there are $\sigma\in (0,r)$, $\tau\in (0,r-\sigma)$ such that $(A_{\sigma+\tau})_{-\tau}\backslash A_{\sigma}$
	contains an inner point.
\end{lemma}
\begin{proof}
	Let $\reach(A)<r$. Then there is $x\in A_{u}\backslash A$ for some $u\in(0,r)$ and $y,z\in A$, $y\not=z$ with $y,z\in\tilde\xi_{A}(x)$. Let $\abs{x-y}=\abs{x-z}=\vcentcolon t_{0}$ then $0<t_{0}<r$.\\
	\textbf{Case 1}
		Let $x\in A_{t_{0}}^{\circ}.$\footnote{At first glance it might seem rather strange that $\dist(x,A)=t_{0}$ and $x\in A_{t_{0}}^{\circ}$, but it is seen easily that this is indeed possible, for example for 
		$A=\partial B_{1}(0)$, $x=0$ and $t_{0}=1$.} 
		Then $\dist(x,\partial [A_{t_{0}}])>0$ and $B_{\dist(x,\partial[A_{t_{0}}])+\epsilon}(x)\subset A_{t_{0}+\epsilon}$ for all $\epsilon>0$. 
		Choose $0<\epsilon<r-t_{0}$ and $0<\delta<\min\{2^{-1}\dist(x,\partial[ A_{t_{0}}]), 2^{-1}t_{0}\}$. Then $B_{\delta}(x)\subset (A_{t_{0}+\epsilon})_{-(\epsilon+\delta)}$ and 
		for all $w\in B_{\delta}(x)$ holds
		\begin{align*}
			\dist(w,A)\geq \dist(x,A)-\abs{x-w}=t_{0}-\abs{x-w}>t_{0}-\delta
		\end{align*}
		so that $B_{\delta}(x)\cap A_{t_{0}-\delta}=\emptyset$. Hence, $x$ is an inner point of $(A_{t_{0}+\epsilon})_{-(\epsilon+\delta)}\backslash A_{t_{0}-\delta}$, 
		i.e. the proposition holds for $\sigma=t_{0}-\delta$ and $\tau=\epsilon+\delta$.\\
	\textbf{Case 2}
		Let $x\in \partial [A_{t_{0}}]$.
		Without loss of generality we might assume that $y=-a e_{1}$, $z=ae_{1}$ and $x=be_{2}$ with $t_{0}^{2}=a^{2}+b^{2}$ and $a>0$. Let $\epsilon\in (0,\min\{r-t_{0},t_{0}\})$.
		Then $B_{\epsilon}(x)\subset (B_{t_{0}+\epsilon}(y)\cap B_{t_{0}+\epsilon}(z))$ and the only elements of $\partial B_{t_{0}+\epsilon}(y)\cap \partial B_{\epsilon}(x)$ 
		and $\partial B_{t_{0}+\epsilon}(z)\cap \partial B_{\epsilon}(x)$ are $x+\epsilon (x-y)/t_{0}$ and $x+\epsilon (x-z)/t_{0}$, respectively. If these two points do not belong
		to $\partial [A_{t_{0}+\epsilon}]$ then $\dist(x,\partial [A_{t_{0}+\epsilon}])>\epsilon$. Now,
		\begin{align*}
			\MoveEqLeft
			\big|x+\epsilon \frac{x-y}{t_{0}}-z\big|^{2}
			=\big| (1+\epsilon/t_{0})be_{2}-(1-\epsilon/t_{0})ae_{1}\big|^{2}\\
			&=(1+\epsilon/t_{0})^{2}b^{2}+(1-\epsilon/t_{0})^{2}a^{2}
			=(1+\epsilon/t_{0})^{2}(a^{2}+b^{2})-4\epsilon a^{2}/t_{0}\\
			&=(t_{0}+\epsilon)^{2}-4\epsilon a^{2}/t_{0}<(t_{0}+\epsilon)^{2},
		\end{align*}
		and hence $x+\epsilon (x-y)/t_{0}\in B_{t_{0}+\epsilon}(z)$ and, by interchanging $y$ and $z$ we obtain $x+\epsilon (x-z)/t_{0}\in B_{t_{0}+\epsilon}(y)$. 
		Hence, we have shown that $x$ lies in the interior of $(A_{t_{0}+\epsilon})_{-\epsilon}$.
		This means that there is $\delta>0$ such that $B_{\delta}(x)\subset (A_{t_{0}+\epsilon})_{-\epsilon}$.
		Now, $B_{\delta}(x)\backslash A_{t_{0}}$ is open and non-empty, 
		as $x\in \partial[A_{t_{0}}]$,\footnote{Note, that we had to distinguish the different cases, because we need $B_{\delta}(x)\backslash A_{t_{0}}$ to be non-empty.} 
		so that there must be $w\in B_{\delta}(x)\backslash A_{t_{0}}$ and 
		$\delta'>0$ with $B_{\delta'}(w)\subset B_{\delta}(x)\backslash A_{t_{0}}$. Therefore $w$ is an inner point of $(A_{t_{0}+\epsilon})_{-\epsilon}\backslash A_{t_{0}}$.
		That is, we have shown the proposition for $\sigma=t_{0}$ and $\tau=\epsilon$.
\end{proof}

\begin{lemma}[(If $\reach(\overline{\R^{n}\backslash A})<r$ then  $A_{-\sigma}\backslash (A_{-(\sigma+\tau)})_{\tau}$ contains an inner point)]\label{differencesetcontainsinnerpoints2}
	Let $A\subset \R^{n}$ be closed, $A\not\in\{\emptyset,\R^{n}\}$ and $\reach(\overline{\R^{n}\backslash A})<r$. Then there are $\sigma\in (0,r)$, $\tau\in (0,r-\sigma)$ 
	such that $A_{-\sigma}\backslash (A_{-(\sigma+\tau)})_{\tau}$ contains an inner point.
\end{lemma}
\begin{proof}
	Let $\reach(\overline{\R^{n}\backslash A})<r$. Then there is $x\in (\overline{\R^{n}\backslash A})_{u}\backslash \overline{\R^{n}\backslash A}\subset A$ for some $u\in(0,r)$ 
	and $y,z\in \partial A$, $y\not=z$ with $y,z\in\tilde\xi_{\partial A}(x)$. Let $\abs{x-y}=\abs{x-z}=\vcentcolon t_{0}$ then $0<t_{0}<r$. 
	Hence, $x$ is an inner point of $A_{-(t_{0}-\delta)}$ for $\delta\in (0,t_{0})$ and consequently $B_{\delta}(x)\subset A_{-(t_{0}-\delta)}$, since
	\begin{align*}
		\dist(w,\partial A)\geq \dist(x,\partial A)-\abs{x-w}\geq t_{0}-\delta
	\end{align*}
	holds for all $w\in B_{\delta}(x)$.
	In the same manner as in Lemma \ref{differencesetcontainsinnerpoints} Case 2 we can show  that for every small enough $\epsilon>0$ we have
	$\dist(x,A_{-(t_{0}+\epsilon)})>\epsilon$. Now, fix $\delta= \min\{\frac{\dist(x,A_{-(t_{0}+\epsilon)})-\epsilon}{3},\frac{t_{0}}{2}\}$, i.e. especially $\epsilon+3\delta\leq \dist(x,A_{-(t_{0}+\epsilon)})$. Then 
	\begin{align*}
		\dist(w,A_{-(t_{0}+\epsilon)})\geq \dist(x,A_{-(t_{0}+\epsilon)})-\abs{x-w}
		\geq \epsilon+2\delta
	\end{align*}
	holds for all $w\in B_{\delta}(x)$. This means we have
	\begin{align*}
		w\not\in (A_{-(t_{0}+\epsilon)})_{\delta+\epsilon}=(A_{-(t_{0}-\delta+\delta+\epsilon)})_{\delta+\epsilon}
	\end{align*}
	for all $w\in B_{\delta}(x)$, or in other words $x$ is an inner point of $A_{-(t_{0}-\delta)}\backslash(A_{-(t_{0}-\delta+\delta+\epsilon)})_{\delta+\epsilon}$ and thus
	we have proven the proposition for $\sigma=t_{0}-\delta$ and $\tau=\delta+\epsilon$.
\end{proof}

\subsection{Closed hypersurfaces of positive reach are $C^{1,1}$ manifolds}

\begin{proposition}[(Normal cones of closed hypersurfaces of positive reach are lines)]\label{uniquenormalcones}
	Let $A$ be a closed hypersurface in $\R^{n}$ and $\reach(A)>0$. Then for $a\in A$ there is a direction $s\in \mathbb{S}^{n-1}$ such that
	\begin{align}\label{definitionouternormal}
		\Nor_{a}A=\R s\quad\text{and}\quad
		\Nor_{a}\overline{\interior(A)}=[0,\infty)s.
	\end{align}
\end{proposition}
\begin{proof}
	Clearly $\xi_{A}(x)=a$, $x\not=a$ implies $B_{\abs{x-a}}(x)\cap A=\emptyset$.
	By Lemma \ref{projectionissurjectiveonUbackslashA} and \eqref{reachAminreachextint} we know that for all $a\in A$ there are $x_{1}\in\interior(A)$, $x_{2}\in\exterior(A)$, such that
	$\xi_{A}(x_{i})=a$ and hence $B_{\abs{x_{1}-a}}(x_{1})\subset \interior(A)$, $B_{\abs{x_{2}-a}}(x_{2})\subset \exterior(A)$. 
	Then we must have that $x_{1},x_{2},a$ lie on a straight line, with $a$ between $x_{1}$ and $x_{2}$, as else $\abs{x_{1}-x_{2}}<\abs{x_{1}-a}+\abs{a-x_{2}}$, so that there would be a point 
	\begin{align*}
		y=x_{1}+\alpha\frac{x_{2}-x_{1}}{\abs{x_{2}-x_{1}}}=x_{2}+(\abs{x_{1}-x_{2}}-\alpha)\frac{x_{1}-x_{2}}{\abs{x_{1}-x_{2}}}\in B_{\abs{x_{1}-a}}(x_{1})\cap B_{\abs{x_{2}-a}}(x_{2})
	\end{align*}
	with $0\leq \alpha< \abs{x_{1}-a}$ and $0\leq\abs{x_{1}-x_{2}}-\alpha< \abs{x_{2}-a}$. Obviously this contradicts $\interior(A)\cap\exterior(A)=\emptyset$.
	Therefore, $\R(x_{1}-a)\subset\Nor_{a}A$, by \eqref{normalconefederer12}, and with the same argument as above we can also show that $\Nor_{a}A\subset \R(x_{1}-a)$.
\end{proof}

An $s\in\mathbb{S}^{n-1}$ with $[0,\infty)s\subset\Nor_{a}\overline{\interior(A)}$ is called \emph{outer normal} of a closed hypersurface $A$ at $a$ and correspondingly $-s$ an \emph{inner normal}.
If the outer normal is unique we denote it by $\nu(a)$.

\begin{lemma}[(Normals are continuous)]\label{normalscontinuous}
	Let $A$ be a closed hypersurface in $\R^{n}$, $\reach(A)>0$, $a_{k}\in A$, $a_{k}\to a$ and $s_{k}\in \Sphere^{n-1}$ be outer normals for $A$ at $a_{k}$. 
	Then $s_{k}\to s$ and $s\in\Sphere^{n-1}$ is the outer normal of $A$ at $a$.
\end{lemma}
\begin{proof}
	Let $(s_{k_{l}})_{l\in\N}$ be a subsequence. Then, as $\Sphere^{n-1}$ is compact, there is an $u\in \Sphere^{n-1}$ and a further subsequence with $s_{k_{l_{m}}}\to u$.
	Since $\xi_{A}$ is continuous, see \cite[4.8 Theorem (4)]{Federer1959a}, we have
	\begin{align*}
		a_{k_{l_{m}}}=\xi_{A}(a_{k_{l_{m}}}+ts_{k_{l_{m}}})\to a=\xi_{A}(a+tu)\quad\text{for all }t<\reach(A).
	\end{align*}
	According to \cite[4.8 Theorem (2)]{Federer1959a} holds $u\in \Nor_{a}A$.
	By Proposition \ref{uniquenormalcones} there is a single $s\in\Sphere^{n-1}$ such that $u=s$ for all subsequences and $s$ is outer normal of $A$ at $a$. By Urysohn's principle we have $s_{k}\to s$.
\end{proof}

The proof also shows that for any closed set of positive reach the limit of normals is a normal at the limit point.

\begin{lemma}[(Closed hypersurface of positive reach is locally a graph)]\label{uniquenormallocallygraph}
	Let $A\subset \R^{n}$ be a closed hypersurface, $\reach(A)>0$, $a\in A$ such that $\Nor_{a}A=\R s$ and $s\in\Sphere^{n-1}$ is an outer normal. Then $A$ is locally a graph over $a+(\Nor_{a}A)^{\perp}$.
	Put more precisely, this means that there is $\epsilon>0$ such that after a rotation and translation $\Phi:\R^{n}\to\R^{n}$, transforming $a$ to $0$ and $s$ to $e_{n}$, we can write
	\begin{align*}
		\Psi: \R^{n-1}\supset B_{\epsilon}(0)\to \Phi(B_{\epsilon}(a)\cap A),\,v\mapsto (v,f(v)),
	\end{align*}
	with a bijective function $\Psi$ and a scalar function $f:\R^{n-1}\to \R$.
\end{lemma}
\begin{proof}
	Assume that the proposition is not true. Without loss of generality we might assume $a=0$ and $s=e_{n}$. Then for every $\epsilon>0$ there are $y=y(\epsilon),z=z(\epsilon)\in B_{\epsilon}(0)\cap A$, 
	$y\not= z$ such that $y_{i}=z_{i}$, for $i=1,\ldots,n-1$. Without loss of generality let $0<y_{n}<z_{n}$. If $s_{y}$ is the outer normal at $y$, we know by Lemma \ref{normalscontinuous} that
	$\alpha_{y}\vcentcolon=\measuredangle(s,s_{y})\to 0$, for $\epsilon\to 0$. By elementary geometry we have $y+(0,t)e_{n}\subset B_{t}(y+ts_{y})$, if $\sin(\alpha_{y}/2)\leq 2^{-1}$. 
	This means that $z\in B_{t}(y+ts_{y})$ for $\abs{y-z}=t<\reach(A)$, if $\epsilon$ is small enough. 
	But as we have seen in the proof of Proposition \ref{uniquenormalcones}, we have $B_{t}(y+ts_{y})\cap A=\emptyset$. \emph{Contradiction}.
\end{proof}

The \emph{subdifferential} of a function $f:\Omega\to\R$, $\Omega\subset\R^{n}$ at $x\in\Omega$ is the set
\begin{align*}
	\partial f(x)\vcentcolon=\Big\{v\in\R^{n}\mid \liminf_{y\to x}\frac{f(y)-f(x)-\langle v,y-x\rangle}{\abs{y-x}}\geq 0\Big\},
\end{align*}
see \cite[Definition 8.3, (a) and 8(4), p.301]{Rockafellar1998a}.\\

The next lemma is a special case of \cite[Proposition 1.4]{Lytchak2005a}.

\begin{lemma}[(Closed hypersurface of positive reach are $C^{1,1}$)]\label{closedhypersurfacesareC1}
	Let $A\subset \R^{n}$ be a closed hypersurface, $\reach(A)>0$. Then $A$ is a $C^{1,1}$ hypersurface.
\end{lemma}
\begin{proof}
	\textbf{Step 1}
		From Lemma \ref{uniquenormallocallygraph} we know that we can write $A$ locally as the graph of a real function $f$. Let $a\in A$. 
		Without loss of generality we assume that $s=-e_{n}$ is the, thanks to Lemma \ref{uniquenormalcones}, unique outer normal of $a$ at $A$ and $a=(x,f(x))$.
		By the characterisation of subdifferentials in terms of normal vectors \cite[8.9 Theorem, p.304f,]{Rockafellar1998a} it is clear that $\partial f(x)=\{v\}$, where $(v,-1)\in \Nor_{(x,f(x))}\epi(f)=[0,\infty)s$,
		corresponding to the normal of $\overline{\interior(A)}$. Likewise $\partial (-f)(x)=\{-v\}$, where $(-v,-1)\in \Nor_{(x,-f(x))}\epi(-f)=[0,\infty)s$, 
		corresponding to the normal of $\overline{\exterior(A)}$. This means that
		\begin{align*}
			\lim_{y\to x}\frac{f(y)-f(x)-\langle v,y-x\rangle}{\abs{y-x}}=0,
		\end{align*}
		as $\liminf_{y\to x}[-g(x)]=-\limsup_{y\to x}g(x)$. Hence, $f$ is differentiable at $x$ and $\nabla f(x)=v$. 
		Let $x_{k}\to x_{0}$ and $(\nabla f(x_{k}),-1)\in \Nor_{(x_{k},f(x_{k}))}\epi(f)=[0,\infty)s_{k}$, $\abs{s_{k}}=1$, $k\in\N_{0}$. Then, as we have seen in Lemma \ref{normalscontinuous}, $s_{k}\to s_{0}$ and
		there are $t_{k}\in [0,\infty)$, such that $(\nabla f(x_{k}),-1)=t_{k}s_{k}$. Additionally,
		\begin{align*}
			t_{0}s_{0}^{n}=-1=t_{k}s_{k}^{n}\quad\Rightarrow\quad
			t_{k}=t_{0}\frac{s_{0}^{n}}{s_{k}^{n}}\to t_{0},
		\end{align*}
		and therefore $\nabla f(x_{k})\to \nabla f(x_{0})$. This means the Jacobian matrix $\nabla f$ is continuous and hence $f$ is locally Lipschitz.\\
	\textbf{Step 2}
		Using \cite[4.18 Theorem]{Federer1959a} and the abbreviations $a\vcentcolon=(x,f(x))$, $b_{\pm}\vcentcolon=(x\pm h,f(x\pm h))$ we can estimate
		\begin{align*}
			\MoveEqLeft
			\Big|\frac{f(x-h)-2f(x)+f(x+h)}{\sqrt{1+\abs{\nabla f (x)}^{2}}}\Big|\\
			&=\Big|\frac{-[f(x)-f(x-h)+\langle\nabla f(x),(x-h)-x\rangle]-[f(x)-f(x+h)+\langle\nabla f(x),(x+h)-x\rangle]}{\sqrt{1+\abs{\nabla f (x)}^{2}}}\Big|\\
			&\leq\Big|\Big\langle \begin{bmatrix} x-h\\f(x-h)\end{bmatrix}-\begin{bmatrix} x\\f(x)\end{bmatrix}, 
			\frac{1}{\sqrt{1+\abs{\nabla f(x)}^{2}}} \begin{bmatrix} \nabla f(x)\\-1\end{bmatrix}\Big\rangle\Big|\\
			&\qquad+\Big|\Big\langle \begin{bmatrix} x+h\\f(x+h)\end{bmatrix}-\begin{bmatrix} x\\f(x)\end{bmatrix}, 
			\frac{1}{\sqrt{1+\abs{\nabla f(x)}^{2}}} \begin{bmatrix} \nabla f(x)\\-1\end{bmatrix}\Big\rangle\Big|\\
			&=\pi_{\Nor_{a}A}(b_{-}-a)+\pi_{\Nor_{a}A}(b_{+}-a)
			=\dist(b_{-}-a,\Tan_{a}A)+\dist(b_{+}-a,\Tan_{a}A)\\
			&\leq \frac{\abs{b_{-}-a}^{2}}{2t}+\frac{\abs{b_{+}-a}^{2}}{2t}\\
			&= \frac{\abs{(x-h)-x}^{2}+\abs{f(x-h)-f(x)}^{2}}{2t}+\frac{\abs{(x+h)-x}^{2}+\abs{f(x+h)-f(x)}^{2}}{2t}\\
			&\leq \frac{c^{2}}{t}\abs{h}^{2},
		\end{align*}
		for $t\leq \reach(A)$. Now, Proposition \ref{characterisationofh\"olderfunctions} implies that $f$ is of class $C^{1,1}$.
\end{proof}

The interesting direction of the next very useful proposition can be found in \cite[Lemma 2.1]{Calabi1970a} for a more general modulus of continuity; but as we are only interested in H\"older continuous derivatives,
we use this specialised version. We also found the proposition formulated in \cite[Lemma 2.1]{Lytchak2005a} in a form very close to the way we present it here. The idea of smoothing the function came from
\cite[Theorem 2.1]{La-Torre2005a}.

\begin{proposition}[(Characterisation of $C^{1,\alpha}_{\mathrm{loc}}$ functions)]\label{characterisationofh\"olderfunctions}
	Let $\Omega\subset \R^{n}$ be open, $f:\Omega\to \R^{m}$ bounded and $0<\alpha\leq 1$. Then the following are equivalent
	\begin{itemize}
		\item
			there are $\rho>0$ and $L>0$ such that for all $x\in\Omega$ holds $f\in C^{1,\alpha}(B_{\rho_{x}}(x),\R^{m})$
			and $[Df\vert_{B_{\rho_{x}}(x)}]_{C^{0,\alpha}}\leq L$, where $\rho_{x}=\min\{\dist(x,\partial\Omega),\rho\}$,
		\item
			there is $C>0$ and $\delta>0$ such that for all $x\in \Omega$ and all $\abs{h}<\delta_{x}=\min\{\dist(x,\partial\Omega),\delta\}$ holds
			\begin{align*}
				\abs{f(x-h)-2f(x)+f(x+h)}\leq C \abs{h}^{1+\alpha}.
			\end{align*}
	\end{itemize}
\end{proposition}
\begin{proof}
	\textbf{Step 1}
		Let $f$ be as requested in the first item. Obviously, it is enough to prove the proposition for $m=1$. Using Taylor's Theorem for Lipschitz functions, Theorem \ref{taylortheoremlipschitz}, we know
		\begin{align*}
			f(x\pm h)-f(x)=\int_{0}^{1}\langle \nabla f(x\pm(1-t)h),\pm h \rangle \dt
		\end{align*}
		for all $\abs{h}<\rho_{x}$, and we obtain
		\begin{align*}
			\MoveEqLeft
			\abs{f(x-h)-2f(x)+f(x+h)}\\
			&=\Big| \int_{0}^{1}\langle \nabla f(x+(1-t)h), h \rangle +\langle \nabla f(x-(1-t)h),- h \rangle \dt\Big|\\
			&\leq \int_{0}^{1}\abs{\langle \nabla f(x+(1-t)h)-\nabla f(x-(1-t)h), h \rangle}  \dt\\
			&\leq \int_{0}^{1}L\abs{[x+(1-t)h]-[x-(1-t)h]}^{\alpha} \abs{h}  \dt\\
			&\leq 2^{\alpha}L\abs{h}^{1+\alpha}.
		\end{align*}
	\textbf{Step 2}
		Now let $f$ be as specified in the second item. We estimate
		\begin{align}\label{estimateseries}
			\begin{split}
				\MoveEqLeft
				\Big|\sum_{k=0}^{\infty}2^{k}(f(x)-2f(x+2^{-(k+1)}h)+f(x+2^{-k}h))\Big|
				\leq C\sum_{k=0}^{\infty}2^{k}(2^{-(k+1)}\abs{h})^{1+\alpha}\\
				&=C2^{-(1+\alpha)}\abs{h}^{1+\alpha}\sum_{k=0}^{\infty}(2^{-\alpha})^{k}<\infty,
			\end{split}
		\end{align}
		so that the series converges uniformly in $(x,h)$ on $U\vcentcolon=\bigcup_{x\in\Omega}\{x\}\times B_{\delta_{x}}(0)$ by Weierstra{\ss}' $M$-Test. As the $l$th partial sum is a telescoping sum, we easily compute
		\begin{align}
			\begin{split}\label{partialsums}
				\MoveEqLeft
				S_{l}(x,h)\vcentcolon=\sum_{k=0}^{l}2^{k}(f(x)-2f(x+2^{-(k+1)}h)+f(x+2^{-k}h))\\
				&=\sum_{k=0}^{l}2^{k}f(x)-\sum_{k=1}^{l+1}2^{k}f(x+2^{-k}h)+\sum_{k=0}^{l}2^{k}f(x+2^{-k}h)\\
				&=(2^{l+1}-1)f(x)-2^{l+1}f(x+2^{-(l+1)}h)+f(x+h)\\
				&=f(x+h)-f(x)-\abs{h}\frac{f(x+2^{-(l+1)}h)-f(x)}{2^{-(l+1)}\abs{h}}.
			\end{split}
		\end{align}
		Therefore for all $(x,h)\in U$, $h\not= 0$ the following limit exists (but might depend not only on the direction, but also on the absolute value of $h$)
		\begin{align*}
			\lim_{l\to\infty}\frac{f(x+2^{-(l+1)}h)-f(x)}{2^{-(l+1)}\abs{h}}.
		\end{align*}
	\textbf{Step 3}
		Let $x\in \Omega$ and $y,z\in B_{\delta_{x}/8}(x)$, $y\not=z$. Clearly $B_{\delta_{x}/8}(x)\subset B_{\delta_{z}/2}(z)$ and $\delta_{x}/2\leq \delta_{z}\leq 2\delta_{x}$. Then there is $l\in \N_{0}$ with
		$\delta_{z}/2\leq 2^{l+1}\abs{y-z}<\delta_{z}$. According to \eqref{partialsums} we have
		\begin{align*}
			\MoveEqLeft
			\abs{S_{l}(z,2^{l+1}(y-z))}\\
			&=\Big|f(z+2^{l+1}(y-z))-f(z)-2^{l+1}\abs{y-z}\frac{f(z+2^{-(l+1)}2^{l+1}(y-z))-f(z)}{2^{-(l+1)}2^{l+1}\abs{y-z}}\Big|\\
			&=\Big|f(z+2^{l+1}(y-z))-f(z)-2^{l+1}\abs{y-z}\frac{f(y)-f(z)}{\abs{y-z}}\Big|		
		\end{align*}
		and \eqref{estimateseries} yields
		\begin{align*}
			\MoveEqLeft
			\abs{S_{l}(z,2^{l+1}(y-z))}
			\leq C2^{-(1+\alpha)}\abs{2^{l+1}(y-z)}^{1+\alpha}\sum_{k=0}^{\infty}(2^{-\alpha})^{k}\\
			&\leq \Big(C2^{-(1+\alpha)}\abs{2^{l+2}\delta_{x}}^{\alpha}\sum_{k=0}^{\infty}(2^{-\alpha})^{k}\Big) 2^{l+1}\abs{y-z}
			=\vcentcolon c 2^{l+1}\abs{y-z}.
		\end{align*}
		Now, we use the reverse triangle inequality and the boundedness of $f$, i.e. $\abs{f(x)}\leq M$ for all $x\in \Omega$, to obtain
		\begin{align*}
			\Big|\frac{f(y)-f(z)}{\abs{y-z}}\Big|\leq c + \Big|\frac{f(z+2^{l+1}(y-z))-f(z)}{2^{l+1}\abs{y-z}}\Big|\leq c+2M\frac{2}{\delta_{z}}\leq c+\frac{8M}{\delta_{x}},
		\end{align*}
		so that $f$ is locally Lipschitz.\\
	\textbf{Step 4}
		In retrospect of Step 2 and Step 3 we know that the mapping 
		\begin{align*}
			g_{i}(x,\lambda)\vcentcolon=\lim_{l\to\infty}\frac{f(x+2^{-(l+1)}\lambda e_{i})-f(x)}{2^{-(l+1)}\lambda},\quad i=1,\ldots, n
		\end{align*}
		is continuous on 
		\begin{align*}
			\bigcup_{x\in\Omega}\{x\}\times ([-\delta_{x},\delta_{x}]\backslash \{0\}),
		\end{align*}
		thanks to the uniform limit theorem. Let $f_{\epsilon}$ be the mollification of $f$, i.e. the convolution with standard mollifiers $\eta_{\epsilon}$. Fix $x\in \Omega$ and $0<\abs{\lambda}<\delta_{x}/9$.
		We now want to show that there is a sequence $\epsilon_{k}\downarrow 0$ such that for all $0<\abs{\lambda}<\delta_{x}/9$ 
		we have $g_{i}(x,\lambda)=\lim_{k\to\infty}\partial _{i}f_{\epsilon_{k}}(x)$, regardless of the value of $\lambda$. Since $g_{i}(x,\lambda)$ equals $\partial_{i}f(x)$ at 
		every point $x\in \Omega$ where $f$ is differentiable, which is almost every point of $\Omega$,
		we know by elementary properties of mollifications on Sobolev spaces, note $C^{0,1}\subset W^{1,\infty}$, that
		\begin{align*}
			\partial_{i}f_{\epsilon}(x)=\Big(\eta_{\epsilon}*\lim_{l\to\infty}\Big(\frac{f(\cdot+2^{-l}\lambda e_{i})-f(\cdot)}{2^{-l}\lambda}\Big)\Big)(x)=(\eta_{\epsilon}*g_{i}(\cdot,\lambda))(x),
		\end{align*}
		for all $0<\abs{\lambda}<\delta_{x}/9$ and $\epsilon$ small enough. As $\partial_{i }f_{\epsilon}(x)$ is bounded in $\epsilon$, because $f$ is Lipschitz continuous, 
		there is a sequence $\epsilon_{k}\downarrow 0$ 
		such that $\lim_{k\to\infty}\partial_{i}f_{\epsilon_{k}}(x)=a_{i}$, or in other words, for every $\tilde\epsilon>0$ there is $N_{1}=N_{1}(\tilde\epsilon)$, 
		with $\abs{a_{i}-\partial_{i} f_{\epsilon_{k}}(x)}\leq 2^{-1}\tilde\epsilon$ for all $k\geq N_{1}$. On the other hand we find $N_{2}=N_{2}(\tilde\epsilon)$, such that
		\begin{align*}
			\MoveEqLeft
			\abs{\partial_{i}f_{\epsilon_{k}}(x)-g_{i}(x,\lambda)}=\abs{\eta_{\epsilon_{k}}(x)*g_{i}(x,\lambda)-g_{i}(x,\lambda)}\leq 2^{-1}\tilde\epsilon
		\end{align*}
		for all $k\geq N_{2}$, because $g_{i}(x,\lambda)$ is continuous. Putting the inequalities together we obtain
		\begin{align*}
			\abs{a_{i}-g_{i}(x,\lambda)}\leq \abs{a_{i}-\partial_{i}f_{\epsilon_{k}}(x)}+\abs{\partial_{i}f_{\epsilon_{k}}(x)-g_{i}(x,\lambda)}\leq \tilde\epsilon
		\end{align*}
		for all $k\geq \max\{N_{1},N_{2}\}$, i.e. $g_{i}(x,\lambda)=a_{i}$. By \eqref{estimateseries} and \eqref{partialsums}
		this means $\abs{f(x+\lambda e_{i})-f(x)-a_{i}\lambda}\leq C \abs{\lambda}^{1+\alpha}$, so that $f$ is partially differentiable 
		at $x$ with $\partial_{i} f(x)=a_{i}=g_{i}(x,\lambda)$ with continuous partial derivatives. Therefore $f$ is differentiable.\\
	\textbf{Step 5}
		Let $x\in \Omega$ and $y,z\in B_{\delta_{x}/8}(x)$, $y\not=z$ as in Step 3. Then
		\begin{align*}
			\MoveEqLeft
			\abs{\lim_{l\to\infty}S_{l}(z,y-z)+\lim_{l\to\infty}S_{l}(y,z-y)}\\
			&=\abs{f(y)-f(z)-(y-z)\nabla f(z)+f(z)-f(y)-(z-y)\nabla f(y)}\\
			&=\abs{y-z}\abs{\nabla f(y)-\nabla f(z)}
		\end{align*}
		and \eqref{estimateseries} yields
		\begin{align*}
			\abs{\lim_{l\to\infty}S_{l}(z,y-z)+\lim_{l\to\infty}S_{l}(y,z-y)}
			\leq 2 C2^{-(1+\alpha)}\abs{y-z}^{1+\alpha}\sum_{k=0}^{\infty}(2^{-\alpha})^{k}
			=\vcentcolon \tilde C\abs{y-z}^{1+\alpha}.
		\end{align*}
\end{proof}

\subsection{Closed $C^{1,1}$ hypersurfaces have positive reach}

It is folklore that compact $C^{1,1}$ submanifolds have positive reach and in fact this can even be found in many remarks in the literature, see for example \cite[below 2.1 Definitions]{Fu1989a} or
\cite[under Theorem 1.1]{Lytchak2004b}, but, unfortunately, the author was not able to locate a single proof. Therefore we show the statement in a special case, adapted to our needs.

\begin{lemma}[(Closed $C^{1,1}$ hypersurfaces have positive reach)]
	Let $A\subset \R^{n}$ be a closed hypersurface of class $C^{1,1}$. Then $\reach(A)>0$.
\end{lemma}
\begin{proof}
	As $A$ is $C^{1,1}$ it can be locally written as a graph of a $C^{1,1}$ function. By compactness of $A$ and Lebesgue's Number Lemma we find $\epsilon,\delta>0$ and a finite number $N$ of functions
	$f_{k}\in C^{1,1}(B_{\epsilon}(0),\R)$, $k=1,\ldots, N$, $B_{\epsilon}(0)\subset \R^{n-1}$ such that for every $a\in A$ the set $A\cap B_{\delta}(a)$ is, 
	after a translation and rotation, covered by the graph of a single $f_{k}$.\\
	\textbf{Step 1}
		Let $u,v\in A$ with $\abs{u-v}\leq \delta$. Then both points lie in the graph of a function $f=f_{k}$ and we can write $u=(x,f(x))$, $v=(y,f(y))$ for $x,y\in B_{\epsilon}(0)$.
		The distance of $v-u$ to $\Tan_{u}A$ is given by the projection of $v-u$ on the normal space $\Nor_{u}A$, i.e.
		\begin{align*}
			\MoveEqLeft
			\dist(v-u,\Tan_{u}A)
			=\Big|\Big\langle \begin{bmatrix} y\\f(y)\end{bmatrix}-\begin{bmatrix} x\\f(x)\end{bmatrix}, 
			\frac{1}{\sqrt{1+\abs{\nabla f(x)}^{2}}} \begin{bmatrix} \nabla f(x)\\-1\end{bmatrix}\Big\rangle\Big|\\
			&=\Big|\frac{(y-x)\nabla f(x)-(f(y)-f(x))}{\sqrt{1+\abs{\nabla f(x)}^{2}}}\Big|
			=\Big|\frac{f(x)-f(y)-\nabla f(x)(x-y)}{\sqrt{1+\abs{\nabla f(x)}^{2}}}\Big|.
		\end{align*}
		By Taylor's Theorem for Lipschitz functions, Theorem \ref{taylortheoremlipschitz}, we can write
		\begin{align*}
			f(x)=f(a)+\nabla f(x)\cdot (x-a)+\int_{0}^{1}(1-s)(x-a)^{T}[\mathrm{Hess}f(a+s(x-a))] (x-a)\ds
		\end{align*}
		and estimate
		\begin{align*}
			\MoveEqLeft
			\dist(v-u,\Tan_{u}A)
			\leq\Big|\int_{0}^{1}(1-s)(x-y)^{T}[\mathrm{Hess}f(y+s(x-y))] (x-y)\ds\Big|\\
			&\leq \norm{\mathrm{Hess}f}_{L^{\infty}(B_{\epsilon}(0))}\abs{x-y}^{2}\leq \norm{\mathrm{Hess}f}_{L^{\infty}(B_{\epsilon}(0))}\abs{v-u}^{2}.
		\end{align*}
	\textbf{Step 2}
		Let $u,v\in A$ with $\abs{u-v}> \delta$. Then $\dist(v-u,\Tan_{u}A)\leq \diam(A)<\infty$, so that
		\begin{align*}
			\dist(v-u,\Tan_{u}A)\leq \diam(A)\leq \frac{\diam(A)}{\delta^{2}}\abs{u-v}^{2}.
		\end{align*}
	\textbf{Step 3}
		All in all we have shown
		\begin{align*}
			\dist(v-u,\Tan_{u}A)\leq \max\Big\{\frac{\diam(A)}{\delta^{2}},\norm{\mathrm{Hess}f_{k}}_{L^{\infty}(B_{\epsilon}(0))}\mid k=1,\ldots, N\Big\}\abs{u-v}^{2},
		\end{align*}
		for all $u,v\in A$. Now the proposition follows with \cite[4.18 Theorem]{Federer1959a}.
\end{proof}

\begin{theorem}[(Taylor's theorem for Sobolev functions)]\label{taylortheoremsobolev}
	Let $I\subset \R$ be a bounded open interval, $k\in\N$. Then for all $f\in W^{k,1}(I)$ and $x,a\in I$ holds
	\begin{align*}
		f(x)=\sum_{i=0}^{k-1}\frac{f^{(i)}(a)}{i!}(x-a)^{i}+\int_{a}^{x}\frac{f^{(k)}(t)}{(k-1)!}(x-t)^{k-1}\dt.
	\end{align*}
\end{theorem}
\begin{proof}
	We can follow the usual proof by induction using the fundamental theorem of calculus and integration by parts. This is possible, because the product rule, and therefore integration by parts, also holds for
	absolutely continuous, and hence $W^{1,1}$, functions, see \cite[formula (3.4), p.167]{Heinonen2007a}.
\end{proof}

\begin{theorem}[(Taylor's theorem for Lipschitz functions)]\label{taylortheoremlipschitz}
	Let $\Omega\subset \R^{n}$ be open, $k\in\N_{0}$. Then for all $f\in C^{k,1}(\Omega)$ and $x,a\in \Omega$ with $x+[0,1](a-x)\subset \Omega$ holds
	\begin{align*}
		f(x)=\sum_{\abs{\alpha}=0}^{k}\frac{D^{\alpha}f(a)}{\alpha!}(x-a)^{\alpha}+\sum_{\abs{\beta}=k+1}\frac{k+1}{\beta!}\int_{0}^{1}(1-t)^{k}D^{\beta}f(a+t(x-a))(x-a)^{\beta}\dt.
	\end{align*}
\end{theorem}
\begin{proof}
	We always have $C^{k,1}\subset W^{k+1,\infty}$, so that we can use the standard proof that applies Taylor's Theorem in dimension one, Theorem \ref{taylortheoremsobolev}, 
	to $g=f\circ h$ for $h:[0,1]\to\Omega$ with $h(t)=a+t(x-a)$.
	For this it is important that $g\in W^{k,1}([0,1])$, which is clear as $f$ and $h$ are both $C^{k,1}$, hence $g\in W^{k+1,\infty}([0,1])$, and that $[0,1]$ is bounded.
\end{proof}

\section{Steiner formula and sets of positive reach}

\begin{proof}[Proof of Theorem \ref{theoremsteinerformulaiffpoasitivereach}]
		The equivalence of the last three items is Theorem \ref{equivalencepositivereachC11}, Lemma \ref{propertiesofparallelsets} and Lemma \ref{alternativecharacterisationofreach2}
		together with \eqref{reachAminreachextint}.\\
	\textbf{Step 1}
		Let
		\begin{align}\label{SteinerFormulaFederer}
			V((A_{s})_{t})=\sum_{k=0}^{n}{n\choose k}W_{k}(A_{s})t^{k}
		\end{align}
		for all $s\in (-r,r)$ and $-r<s+t<r$. We compute
		\begin{align}\label{computationformulaquermassintegrals}
			\begin{split}
				\MoveEqLeft
				V(A_{s+t})=\sum_{k=0}^{n}{n\choose k}W_{k}(A)(s+t)^{k}
				=\sum_{k=0}^{n}{n\choose k}W_{k}(A)\sum_{i=0}^{k}{k\choose i}s^{k-i}t^{i}\\
				&=\sum_{k=0}^{n}\sum_{i=0}^{k}{n\choose k}{k\choose i}W_{k}(A)s^{k-i}t^{i}
				=\sum_{i=0}^{n}\sum_{k=i}^{n}{n\choose k}{k\choose i}W_{k}(A)s^{k-i}t^{i}\\
				&=\sum_{i=0}^{n}\sum_{k=i}^{n}{n\choose i}{n-i\choose k-i}W_{k}(A)s^{k-i}t^{i}
				=\sum_{i=0}^{n}{n\choose i}\Big(\sum_{k=i}^{n}{n-i\choose k-i}W_{k}(A)s^{k-i}\Big)t^{i}.
			\end{split}
		\end{align}
		By Lemma \ref{propertiesofparallelsets} holds $V((A_{s})_{t})=V(A_{s+t})$ for $s,t>0$ or $s,t<0$ with $\abs{s+t}<r$,
		so that comparing \eqref{SteinerFormulaFederer} with \eqref{computationformulaquermassintegrals} yields 
		\begin{align}\label{forulaforquermasintegrals}
			W_{i}(A_{s})=\sum_{k=i}^{n}{n-i\choose k-i}W_{k}(A)s^{k-i}.
		\end{align}
		According to Lemma \ref{propertiesofparallelsets} we either have $A_{s+t}\subset (A_{s})_{t}$ or $(A_{s})_{t}\subset A_{s+t}$ for $s\in (-r,r)$, $-r<s+t<r$. 
		By \eqref{computationformulaquermassintegrals} we obtain 
			\begin{align}\label{volumeisequal}
			\begin{split}
				\MoveEqLeft
				V((A_{s})_{t})=\sum_{i=0}^{n}{n\choose i}W_{i}(A_{s})t^{i}
				=\sum_{i=0}^{n}{n\choose i}\Big(\sum_{k=i}^{n}{n-i\choose k-i}W_{k}(A)s^{k-i}\Big)t^{i}
				=V(A_{s+t}),
			\end{split}
		\end{align}
		for $s\in (-r,r)$ and $-r<s+t<r$.
		Assume $\reach(\partial A)<r$. Then the reach of $\overline{\interior(\partial A)}=A$ or $\overline{\exterior(\partial A)}=\overline{\R^{n}\backslash A}$ is strictly smaller than $r$. 
		Now, we obtain a contradiction to \eqref{volumeisequal} via Lemma \ref{differencesetcontainsinnerpoints} for $s=\sigma+\tau$, $t=-\tau$ if $\reach(A)<r$ and via
		Lemma \ref{differencesetcontainsinnerpoints2} for $s=-(\sigma+\tau)$, $t=\tau$ in case $\reach(\overline{\R^{n}\backslash A})<r$.\\
	\textbf{Step 2}
		Let the last three items hold. Then according to the second item of Lemma \ref{paralellsurfaces} for $B=A$, $s=t$ and \eqref{reachAminreachextint} we have 
		$\reach(A_{s})\geq \reach(\partial A_{s})\geq r-\abs{s}$ for $s\in (-r,r)$.
		Using Federer's Steiner formula for sets of positive reach, see \cite[5.6 Theorem]{Federer1959a}, we obtain \eqref{SteinerFormulaFederer} for all $s\in (-r,r)$ and $0<t<r-\abs{s}$
		and, obviously, this also holds for $t=0$.
		In a first part we use this to prove \eqref{SteinerFormulaFederer} for $s\in (-r,r)$ and $s\leq s+t<r$. These results are then used in a second part to establish \eqref{SteinerFormulaFederer} 
		for $s\in (-r,r)$ and $-r<s+t<r$.\\
		\textbf{Part 1}
			Making use of Federer's Steiner formula we can do a computation similar to \eqref{computationformulaquermassintegrals} for $V((A_{s+t})_{u})=V((A_{s})_{t+u})$, $0<t<r-\abs{s}$ and $0<u<r-\abs{s}-t$,
			note $t+u>0$, to obtain
			\begin{align}\label{formulaquermassintegralss}
				W_{i}(A_{s+t})=\sum_{k=i}^{n}{n-i\choose k-i}W_{k}(A_{s})t^{k-i}.
			\end{align}
			For $s\in [0,r)$ we already have \eqref{SteinerFormulaFederer} for all $0\leq t<r-s$.
			Let $s\in (-r,0)$. Choose $u\in (0,r-\abs{s})$ and $v\in (0,r-\abs{s+u})$. Now, again using Federer's Steiner formula, we can compute $V((A_{s+u})_{v})$ and substitute \eqref{formulaquermassintegralss}, 
			using the same tricks as in \eqref{computationformulaquermassintegrals} and \eqref{volumeisequal}, to obtain
			\begin{align*}
				\MoveEqLeft
				V((A_{s})_{u+v})=V((A_{s+u})_{v})=\sum_{k=0}^{n}{n \choose k}W_{k}(A_{s+u})v^{k}
				=\sum_{k=0}^{n}{n \choose k}\sum_{j=k}^{n}{n-k \choose j-k}W_{j}(A_{s})u^{j-k}v^{k}\\
				&=\sum_{j=0}^{n}\sum_{k=0}^{j}{j \choose k}{n \choose j}W_{j}(A_{s})u^{j-k}v^{k}
				=\sum_{j=0}^{n}{n \choose j}W_{j}(A_{s})(u+v)^{j}.
			\end{align*}		
			This means we have shown \eqref{SteinerFormulaFederer} for all $s\in (-r,0)$ and $u+v=t\in (0,r-\abs{s}+r-\abs{s+u})$, where
			\begin{align*}
				r-\abs{s}+r-\abs{s+u}=2r-u\geq 2r-(r+s)=r+s
				\quad\text{if }s+u>0
			\end{align*}
			and
			\begin{align*}
				r-\abs{s}+r-\abs{s+u}=2(r+s)+u\geq r+s
				\quad\text{if }s+u\leq 0.
			\end{align*}
			Iteration yields \eqref{SteinerFormulaFederer} for all $s\in (-r,r)$ and $s\leq s+t<r$.\\
		\textbf{Part 2}	
			Let $s\in (-r,r)$ and $-r<s+t<r$. We want to obtain \eqref{SteinerFormulaFederer} for this range of parameters. Choose $0<u$ with $-r<s+t+u<r$ and $0<t+u$.
			As in Part 1 we can use the Steiner formula, now with the extended range from Part 1, to compute $V((A_{s+t})_{u})=V((A_{s})_{t+u})$, which yields \eqref{formulaquermassintegralss} 
			for $s\in (-r,r)$ and $-r<s+t<r$.\footnote{Note that this range could not be covered in the Part 1, because the range of $u$ there is restricted to
			$0<u<r-\abs{s+t}$, so that $V((A_{s+t})_{u})$ can be expanded in $u$ via the Steiner formula. This is why we first had to extend the range to $0<u<r-(s+t)$.}
			This time choose $0<u$ such that $-r<s+t-u$. Then by the Steiner formula from Part 1 holds
			\begin{align*}
				\MoveEqLeft
				V((A_{s})_{t})=V((A_{s+t-u})_{u})=\sum_{k=0}^{n}{n \choose k}\sum_{i=k}^{n}{n-k\choose i-k}W_{i}(A_{s})(t-u)^{i-k}u^{k}\\
				&=\sum_{i=0}^{n}\sum_{k=0}^{i}{i \choose k}{n\choose i}W_{i}(A_{s})(t-u)^{i-k}u^{k}
				=\sum_{i=0}^{n}{n \choose i}W_{i}(A_{s})t^{i}.
			\end{align*}
\end{proof}

\begin{proof}[Proof of Theorem \ref{positivereachandsteinerformula}]
	Except for the differences explained below the proof is the same as for Theorem \ref{theoremsteinerformulaiffpoasitivereach}.
	For the very last part of the analog of Step 1 in the proof of Theorem \ref{theoremsteinerformulaiffpoasitivereach} we assume $\reach(A)<r$ and then obtain a contradiction to
	to \eqref{volumeisequal} via Lemma \ref{differencesetcontainsinnerpoints} for $s=\sigma+\tau$, $t=-\tau$. For the analog of Step 2 it is enough to have $\reach(A)>0$, because we do not
	have to use Lemma \ref{paralellsurfaces}, as we can simply employ \cite[4.9 Corollary]{Federer1959a} to obtain $\reach(A_{s})\geq r-s$ for $s\in (0,r)$. Then we can follow the other steps,
	skipping the middle part, to obtain the desired result.	
\end{proof}

\begin{lemma}[(Parallel surfaces and normals)]\label{paralellsurfaces}
	Let $A$ be a closed hypersurface with $\reach(A)>t>0$. Denote $B\vcentcolon=\overline{\interior(A)}$. 
	\begin{itemize}
		\item
			The mapping $\phi_{t}:A\to  \partial[B_{\pm t}]$, $a\mapsto a\pm t\nu(a)$ is bijective and $\nu(a)=\nu(\phi_{t}(a))$.
		\item
			The boundary $\partial[B_{\pm t}]$ is a $C^{1,1}$ manifold with $\reach(\partial[B_{\pm t}])\geq \reach(A)-t$.
		\item
			If $A$ is the boundary of a convex set with non-empty interior
			we have $\reach(\partial[B_{\pm t}])= \reach(A)\pm t$.
	\end{itemize}
\end{lemma}
\begin{proof}
	That $\phi_{t}$ is injective is a direct consequence of Lemma \ref{alternativecharacterisationreach}.
	On the other hand we have $\xi_{A}(x)=\vcentcolon a\in A$ for every $x\in \partial[B_{\pm t}]$ and hence $x-a\in\Nor_{a}A$, so that
	$x=a+t(x-a)/\abs{x-a}=\phi_{t}(a)$. The coincidence of normals is a consequence of \eqref{normalconefederer12} and \cite[4.9 Corollary]{Federer1959a}. 
	From the alternative characterisation of $\reach$ in Lemma \ref{alternativecharacterisationreach} we infer the estimate for $\reach(\partial[B_{\pm t}])$.
	Now let $A$ be the boundary of a closed convex set $B$ with non-empty interior. As $B_{\pm t}$ is convex, 
	see \cite[\S 6,p.17]{Hadwiger1955a}, it is clear that $\reach(B_{\pm t})=\infty$, so that the formula for $\reach(\partial[B_{\pm t}])$ follows from 
	Lemma \ref{alternativecharacterisationreach} and  \eqref{reachAminreachextint}.
	The $C^{1,1}$ regularity is a consequence of Theorem \ref{equivalencepositivereachC11}.
\end{proof}

\subsection{Hadwiger's Problem}

\begin{proof}[Proof of Theorem \ref{bigequivalence}]
	The equivalence of the first three items is actually shown in \cite[proof of Theorem 1.1]{Hernandez-Cifre2010b} and the equivalence of the last two items is Theorem \ref{theoremsteinerformulaiffpoasitivereach}.\\
	\textbf{Step 1}
		Let $K=(K_{-r})_{r}$ and $x\in B_{r}(\partial K)$. If $x\in \exterior(\partial K)$ we have a unique projection $\xi_{\partial K}(x)$, so let $x\in\interior(\partial K)$.
		We know that $K_{-r}$ is convex and, as $x\in \exterior(\partial(K_{-r}))\cup \partial(K_{-r})$, we have a unique projection $y=\xi_{\partial (K_{-r})}(x)$. Let $\{z\}=[0,\infty)(x-y)\cap \partial K$.
		Then $\xi_{\partial(K_{-r})}(z)=y$ by \eqref{tracingthenormal}, as $K_{-r}$ is convex and hence $\reach(K_{-r})=\infty$. Then $B_{r}(y)\subset K$ and $\abs{z-y}=r$.
		This means $z\in \tilde\xi_{\partial K}(y)$ and consequently $\xi_{\partial K}(x)=z$, see Lemma \ref{projectionproperties}. Therefore $\reach(\partial K)\geq r$.\\
	\textbf{Step 2}
		Let $\reach(\partial K)\geq r$. Then according to Theorem \ref{theoremsteinerformulaiffpoasitivereach}, we have a Steiner formula for every $K_{s}$, $s\in (-r,r)$. This directly yields
		\eqref{forulaforquermasintegrals} and $W_{i}'(s)=(n-i)W_{i+1}(s)$ for the querma{\ss}integrals. Hence $K\in \mathcal{R}_{n-1}(r)$.
\end{proof}

\subsection{Gradient flow of mean breadth}\label{sectiongradientflow}

Before we start to prove \eqref{gradientflowede} in Proposition \ref{gradientflowmeanbreadth} we should at least, very briefly, explain the notation that is specific to gradient flows on metric spaces.
For a curve $x:I\to X$ from an interval $I$ to a metric space $X$ we define the \emph{metric derivative} $\abs{\dot x (t_{0})}$ at a point $t_{0}\in I$ by
\begin{align*}
	\abs{\dot x (t_{0})}\vcentcolon=\lim_{\substack{t\to t_{0}\\t\in I}}\frac{d(x(t),x(t_{0}))}{\abs{t-t_{0}}}
\end{align*}
if this limit exists. The \emph{slope} $\abs{\nabla F}(x_{0})$ of map $F:X\to \R$ at a point $x_{0}\in X$ is set to be
\begin{align*}
	\abs{\nabla F}(x_{0})\vcentcolon=\limsup_{x\to x_{0}}\frac{(F(x_{0})-F(x))_{+}}{d(x_{0},x)},
\end{align*}
where $(a)_{+}\vcentcolon=\max\{a,0\}$ for $a\in\R$. A curve $x:I\to X$ in a metric space $(X,d)$ is called \emph{absolutely continuous} if there is a function
$f\in L^{1}(I)$ such that
\begin{align*}
	d(x(s),x(t))\leq \int_{s}^{t}f(y)\dd y\quad\text{for all }s,t\in I\text{ with }s<t.
\end{align*}

\begin{lemma}[(Computation of the slope $\abs{\nabla W_{i}}$)]\label{computationofslopequermass}
	For all $K\in\mathcal{K}^{1,1}$ we have $\abs{\nabla W_{i}}(K)=(n-i)W_{i+1}(K)$ for $i=0,\ldots, n-1$.
\end{lemma}
\begin{proof}
	Let $t<\reach(\partial K)$. 
	According to \cite[Theorem 6.13 (iv), p.105]{Gruber2007a} the querma{\ss}integrals $W_{i}$ are monotonic with regard to inclusion,
	i.e. for $L\subset K$
	we have $W_{i}(L)\leq W_{i}(K)$. Hence the set in $\overline B_{t}(K)\cap \mathcal{K}^{1,1}$ with least $W_{i}$ is $K_{-t}$. Here $\overline B_{t}(K)$ is the closed ball about
	$K$ with regard to the Hausdorff metric.
	We compute
	\begin{align*}
		\sup_{L\in \overline B_{t}(K)\cap \mathcal{K}^{1,1}}(W_{i}(K)-W_{i}(L))_{+}
		=W_{i}(K)-W_{i}(K_{-t})
	\end{align*}
	and consequently with the help of Theorem \ref{bigequivalence}
	\begin{align*}
		\MoveEqLeft
		\abs{\nabla W_{i}}(K)=\limsup_{\substack{L\to K\\L\in\mathcal{K}^{1,1}}}\frac{(W_{i}(K)-W_{i}(L))_{+}}{\distH(K,L)}
		=\limsup_{t\to 0}\frac{W_{i}(K)-W_{i}(K_{-t})}{t}\\
		&=W_{i}'(0)=(n-i)W_{i+1}(0)
		=(n-i)W_{i+1}(K).
	\end{align*}
	Notice for $\distH(K,L)=t$ is $(W_{i}(K)-W_{i}(L))_{+}\leq W_{i}(K)-W_{i}(K_{-t})$.
\end{proof}

\begin{proof}[Proof of Proposition \ref{gradientflowmeanbreadth}]
	We have $\distH(x(s),x(t))=\omega_{n} (t-s)$ for $s<t$, so that $x$ is absolutely continuous. For $u\in (0,\omega_{n}^{-1}\reach(\partial K))$ holds
	\begin{align*}
		\abs{\dot x(u)}=\lim_{h\to 0}\frac{\distH(x(u+h),x(u))}{\abs{h}}=\frac{\omega_{n}\abs{h}}{\abs{h}}=\omega_{n}.
	\end{align*}
	By Lemma \ref{computationofslopequermass} we already know $\abs{\nabla W_{n-1}}(C)=W_{n}(C)=\omega_{n}$ for all $C\in\mathcal{K}^{1,1}$ and together with 
	\begin{align*}
		W_{n-1}(K_{-t})=W_{n-1}((K_{-t})_{t})-W_{n}(K_{-t})t=W_{n-1}(K)-\omega_{n}t,
	\end{align*}
	from the usual expansion \eqref{forulaforquermasintegrals} of $W_{i}$ with $(K_{-t})_{t}=K$ from the proof of Theorem \ref{theoremsteinerformulaiffpoasitivereach}, we have proven \eqref{gradientflowede}.\\

	Clearly $x(t)\to x(T)$ for $t\to T$ and $x(T)$ is a compact, convex set and hence either contained in a lower dimensional affine subspace or it has non-empty interior. 
	Assume that $x(T)$ has non-empty interior and $\partial x(T)$ has positive reach. Then, by Theorem \ref{equivalencepositivereachC11}, $\partial x(T)$ is of class $C^{1,1}$ 
	and we must have $\nu_{\partial K}(a)=\nu_{\partial K_{-\omega_{n}T}}(a-\omega_{n}T\nu_{\partial K}(a))$ for all $a\in \partial K$. Thus, we obtain a contradiction to $\omega_{n}T=\reach(\partial K)$ in the representation of 
	Lemma \ref{alternativecharacterisationreach}, because there must be an $\epsilon$ neighbourhood 
	of $\partial x(T)$, where the normals cannot intersect, as $\reach(\partial x(T))>0$.
\end{proof}

\begin{appendix}

\section{Querma{\ss}integrals as mean curvature integrals}

\begin{lemma}[(Querma{\ss}integrals as mean curvature integrals)]\label{quermassintegralsasmeancurvatureintegrals}
	Let $A\subset \R^{n}$, $\partial A$ a closed hypersurface with $\reach(\partial A)>0$. Then
	\begin{align}\label{quermassintegralsasmeancurvatureintegral}
		W_{i}(A)=n^{-1}\int_{\partial A}H_{i-1}^{(n-1)}(\kappa_{1},\ldots,\kappa_{n-1})\dd\HM^{n-1},
	\end{align}
	where $H_{j}^{(k)}$ is the $j$th elementary symmetric polynomial in $k$ variables, i.e.
	\begin{align*}
		H_{j}^{(k)}(x_{1},\ldots,x_{k})\vcentcolon={k \choose j}^{-1}\sum_{1\leq l_{1}<\ldots < l_{j}\leq k}x_{l_{1}}\dots x_{l_{j}}
	\end{align*}
	for $j=1,\ldots,k$ and $H_{0}^{(k)}=1$.
\end{lemma}
\begin{proof}
	By Lemma \ref{closedhypersurfacesareC1} we can write $\partial A$ locally as the graph of a function $f\in C^{1,1}(\Omega,\R)$, $\Omega\subset\R^{n-1}$. 
	Note, that the Hessian $\Hess f$ of $f$ is symmetric almost everywhere.
	For $\rho<\reach(\partial A)$ we define the mapping
	\begin{align*}
		\Phi:\Omega\times(0,\rho)\to\R^{n},\,\begin{bmatrix}x\\t\end{bmatrix}\mapsto \begin{bmatrix}x\\f(x)\end{bmatrix}
		+t(1+\abs{\nabla f(x)}^{2})^{-1/2}\begin{bmatrix}\nabla f(x)\\-1\end{bmatrix},
	\end{align*}
	which is bijective onto its image. The vector after the factor $t$ is equal to $\nu((x,f(x)))$. As $f$ and $\nu$ are Lipschitz continuous, the same holds for $\Phi$.
	This means we can extend $\Phi$ to a Lipschitz mapping on the whole $\R^{n-1}$ by Kirszbraun's Theorem \cite[2.10.42 Theorem, p.201]{Federer1969a}
	and then use the area formula \cite[3.2.3 Theorem, p.243]{Federer1969a} to compute
	\begin{align*}
		\HM^{n}(\Phi(\Omega\times (0,\rho) ))=\int_{\Omega\times(0,\rho)}\abs{\det(\text{D}\Phi(y))}\dy.
	\end{align*}
	For the Jacobian matrix $\text{D}\Phi$ we obtain
	\begin{align*}
		\MoveEqLeft
		\text{D}\Phi(x,t)
		=\begin{bmatrix}E_{n-1}&0_{(n-1)\times 1}\\\nabla f(x)^{T}&0_{1\times 1}\end{bmatrix}
		+t\begin{bmatrix}(\partial_{1}\phi(x))\begin{bmatrix}\nabla f(x)\\-1\end{bmatrix}&\dots (\partial_{n-1}\phi(x))\begin{bmatrix}\nabla f(x)\\-1\end{bmatrix}&0_{n\times 1} \end{bmatrix}\\
		&+t\phi(x)\begin{bmatrix}\Hess f(x)&0_{(n-1)\times 1}\\0_{1\times(n-1)}&0_{1\times 1}\end{bmatrix}+\phi(x)\begin{bmatrix}0_{(n-1)\times(n-1)}&\nabla f(x)\\0_{1\times (n-1)}&-1\end{bmatrix}
	\end{align*}
	for almost every $x\in\Omega$, where $E_{k}$ is the identity matrix of size $k$ and we abbreviate $\phi(x)\vcentcolon=(1+\abs{\nabla f(x)}^{2})^{-1/2}$. 
	Now we can use the last column of the last matrix to eliminate the whole second matrix and the last row of the first matrix, to obtain a matrix
	\begin{align*}
		\begin{bmatrix}E_{n-1}+\nabla f(x) [\nabla f(x)]^{T}+t\phi(x)\Hess f(x)&\;\;&\phi(x)\nabla f(x)\\0_{1\times(n-1)}&\;\;&-\phi(x)\end{bmatrix}
	\end{align*}
	with the same determinant as $\text{D}\Phi(x)$. For the surface described by the graph of $f$ the metric tensor is given by $B\vcentcolon= E_{n-1}+\nabla f(x) [\nabla f(x)]^{T}$
	and the curvature tensor by $C\vcentcolon=\phi(x)\Hess f(x)$, note $\det(B)=1+\abs{\nabla f(x)}^{2}=\phi(x)^{-2}$. This means the eigenvalues of $M\vcentcolon= B^{-1}C$ 
	are the principal curvatures $\kappa_{i}$, so that the eigenvalues of $E_{n-1}+t M$ are $1+t\kappa_{i}$. Hence
	\begin{align*}
		\MoveEqLeft
		\det(\text{D}\Phi)=\det \Big(\begin{bmatrix}B+t C&\;\;& \phi\nabla f\\0_{1\times (n-1)}&\;\;&-\phi  \end{bmatrix}\Big)
		=\phi\det \Big(\begin{bmatrix}B(E_{n-1}+t M)&\;\;& \nabla f\\0_{1\times (n-1)}&\;\;&-1  \end{bmatrix}\Big)\\
		&=-\phi\det(B)\det(E_{n-1}+tM)=-\det(B)^{1/2}\prod_{i=1}^{n-1}(1+t\kappa_{i})\\
		&=-\det(B)^{1/2}\Big(\sum_{i=0}^{n-1}{n-1\choose i}H^{(n-1)}_{i}(\kappa_{1},\ldots,\kappa_{n-1})t^{i}\Big).
	\end{align*}	
	Therefore
	\begin{align*}
		\MoveEqLeft
		\HM^{n}(\Phi(\Omega\times (0,\rho) ))=\int_{\Omega\times(0,\rho)}\abs{\det(\text{D}\Phi(y))}\dy\\
		&=\sum_{i=0}^{n-1}\frac{1}{i+1}{n-1\choose i}\int_{\Omega}H^{(n-1)}_{i}(\kappa_{1},\ldots,\kappa_{n-1})\rho^{i+1}\det(B)^{1/2}\dx\\
		&=\sum_{j=1}^{n}{n\choose j}n^{-1}\int_{\mathrm{graph}(f)}H^{(n-1)}_{j-1}(\kappa_{1},\ldots,\kappa_{n-1})\dd\HM^{n-1}\; \rho^{j}.
	\end{align*}
	Adding $\HM^{n}(A)$ and using a covering of $\partial A$ by graphs together with the appropriate partition of unity we obtain
	\begin{align*}
		V(A_{\rho})=\HM^{n}(A_{\rho})=\HM^{n}(A)+\sum_{j=1}^{n}{n\choose j}n^{-1}\int_{\partial A}H^{(n-1)}_{j-1}(\kappa_{1},\ldots,\kappa_{n-1})\dd \HM^{n-1}\; \rho^{j}
	\end{align*}
	comparing this with the Steiner formula \eqref{steinerformulaintroduction} yields \eqref{quermassintegralsasmeancurvatureintegral}.
\end{proof}

\begin{remark}[(Mean breadth for $n=2$ and $n=3$)]
	In the special cases of dimension $n=2$ and $n=3$ the statement of Lemma \ref{quermassintegralsasmeancurvatureintegrals} for $i=n-1$ is
	\begin{align}
			W_{n-1}(K)&=2^{-1}\HM^{1}(\partial K)\quad\text{for }n=2,\nonumber\\
			W_{n-1}(K)&=3^{-1}\int_{\partial K}H\dd\HM^{2}\quad\text{for }n=3,\label{curvaturemeasuren=23}
	\end{align}
	where $H$ is the usual mean curvature. The coefficient $W_{n-1}(K)$ is, at least in the convex case, usually called \emph{mean breadth} of $K$
\end{remark}

\begin{remark}[(Gau{\ss}-Bonnet Theorem for sets of positive reach)]
	Note that the representation of querma{\ss}integrals of sets bounded by hypersurfaces of positive reach as mean curvature integrals, 
	Lemma \ref{quermassintegralsasmeancurvatureintegrals}, easily gives us a Gau{\ss}-Bonnet Theorem for these surfaces
	\begin{align*}
		\int_{\partial A}K_{G}\dd \sigma=nW_{n}(A)=n\omega_{n}\chi(A),
	\end{align*}
	where $K_{G}$ is the Gau{\ss} curvature. Note that here the dimension $n$ does not have to be odd, as in the generalized Gau{\ss}-Bonnet Theorem. 
\end{remark}

\end{appendix}

\bibliography{}{}
\bibliographystyle{amsalpha}
\bigskip
\noindent
\parbox[t]{.8\textwidth}{
Sebastian Scholtes\\
Institut f{\"u}r Mathematik\\
RWTH Aachen University\\
Templergraben 55\\
D--52062 Aachen, Germany\\
sebastian.scholtes@rwth-aachen.de}

\end{document}